\documentclass{amsart}

\usepackage{amsmath, amsthm, amssymb}
\usepackage{enumerate}
\usepackage{verbatim}
\usepackage{esint}
\usepackage{graphicx}
\usepackage{caption}
\usepackage{subcaption}
\usepackage{xcolor}

\title{The Analyst's traveling salesman theorem in graph inverse limits}

\author{Guy C. David}
\address{Courant Institute of Mathematical Sciences, New York University, New York, NY 10012}
\email{guydavid@math.nyu.edu}

\author{Raanan Schul}
\address{Department of Mathematics\\ Stony Brook University\\ Stony Brook, NY 11794-3651}
\email{schul@math.sunysb.edu}

\date{\today} 
\thanks{R.~ Schul was partially supported by NSF  DMS 1361473.}

\begin{document}

\begin{abstract}
We prove a version of Peter Jones' Analyst's traveling salesman theorem in a class of highly non-Euclidean metric spaces introduced by Laakso and generalized by  Cheeger-Kleiner. These spaces are constructed as inverse limits of metric graphs, and include examples which are doubling and have a Poincar\'e inequality. We show that a set in one of these spaces is contained in a rectifiable curve if and only if it is quantitatively ``flat'' at most locations and scales, where flatness is measured with respect to so-called monotone geodesics. This provides a first examination of quantitative rectifiability within these spaces.

\end{abstract}

\maketitle
\newtheorem{thm}{Theorem}[section]
\newtheorem{lemma}[thm]{Lemma}
\newtheorem{prop}[thm]{Proposition}
\newtheorem{cor}[thm]{Corollary}
\newtheorem{claim}[thm]{Claim}

\theoremstyle{remark}
\newtheorem{rmk}[thm]{Remark}

\theoremstyle{definition}
\newtheorem{definition}[thm]{Definition}

\theoremstyle{remark}
\newtheorem{example}[thm]{Example}

\newcommand{\RS}[1]{{\color{red} [RS: #1]}}

\newcommand{\cB}{\mathcal{B}}
\newcommand{\bR}{\mathbb{R}}
\newcommand{\RR}{\mathbb{R}}
\newcommand{\length}{\textnormal{length}}
\newcommand{\dist}{\textnormal{dist}}
\newcommand{\diam}{\textnormal{diam}}
\newcommand{\um}{\underline{m}}
\newcommand{\rad}{\textnormal{radius}}
\newcommand{\bZ}{\mathbb{Z}}
\newcommand{\bH}{\mathbb{H}}
\newcommand{\cH}{\mathcal{H}}

\numberwithin{equation}{section}

\setcounter{tocdepth}{1}
\tableofcontents

\section{Introduction}

The Analyst's traveling salesman theorem (see Theorem \ref{t:Hilbert-tst}  below)  is a strong quantitative, geometric form of the result that finite-length curves have tangents at almost every point. 
It was first proven in the plane by Jones \cite{Jones-TSP}, and then extended to $\bR^d$ by Okikiolu \cite{Ok92} and to Hilbert space by the second author \cite{Sc07_hilbert}.  In order to state it, we need to first define the Jones $\beta$ numbers in the Hilbert space (or Euclidean) setting.
Let $B\subset \cH$ be a ball in a  (possibly infinite dimensional) Hilbert space $\cH$.
Let $E\subset \cH$ be any subset.
Define 
$$\beta_{\cH,E}(B) = \inf\limits_{\ell \textrm{ line}} \sup\limits_{x\in E\cap B}\frac{\dist(x,\ell)}{\diam(B)}\,$$
where the infimum is taken over all lines in $\cH$.

\begin{thm}[Euclidean Traveling salesman, \cite{Jones-TSP,Ok92,Sc07_hilbert}]\label{t:Hilbert-tst}
Let $\cH$ be a (possibly infinite dimensional) Hilbert space.
There is a constant $A_0>1$ such that for any $A>A_0$ there is a constant    $C=C(A)>0$ such that the following holds.
Let $E\subset \cH$.
For integer $n$, let $X_n\subset E$ be a $2^{-n}$ separated net such that $X_n\subset X_{n+1}$.  
Let $\cB=\{B(x,A2^{-n}) : n\in \bZ,\ \ x\in X_n\}$.
\begin{itemize}
 \item[(A)] If $\Gamma\supset E$, and $\Gamma$ is compact and connected then
 $$ \sum_{B\in \cB} \beta_{\cH,\Gamma}(B)^2 \diam(B) \leq C \cH^1(\Gamma). $$
 \item[(B)] There exists $\Gamma\supset E$ such that  $\Gamma$ is compact and connected and
$$ \cH^1(\Gamma) \leq C \left(\diam(E) + \sum_{B\in \cB} \beta_{\cH,E}(B)^2 \diam(B)\right). $$
\end{itemize}
\end{thm}

The original motivation for proving Theorem \ref{t:Hilbert-tst} in the plane in \cite{Jones-TSP} was the study of singular integrals (see e.g. \cite{Jones87-El-Escorial}) and harmonic measure (see e.g. \cite{BJ90}). Significant work in the direction of singular integrals was done by David and Semmes (see for instance \cite{DS-sing-int-book, DS-blue-book} and the references within), as part of a field which now falls under the name ``quantitative rectifiability", to separate it from the older field of ``qualitative rectifiability". Several works tying the two fields exist, most recently by Tolsa \cite{Tolsa15-part-I}, Azzam and Tolsa \cite{Azzam-Tolsa15-part-II}, and Badger and the second author \cite{Badger-Schul-characterization}. (This last reference has a detailed exposition of the fields, to which we refer the interested reader.)  We remark that a  variant of   Theorem \ref{t:Hilbert-tst}  was needed to do the work in   \cite{Badger-Schul-characterization}.

Several papers generalizing the traveling salesman theorem to metric spaces have been written by Hahlomma and the second author.
With an a priori assumption on Ahlfors regularity (i.e, that the metric space supports a measure $\mu$ so that   the growth of $\mu(B(x,r))$ is bounded above and below by linear functions of $r$, as long as $r$ does not exceed the diameter of the space) one has results analogous to Theorem \ref{t:Hilbert-tst}: see  \cite{Hahlomaa-AR, Sc07_metric}.  
To this end, one needs to redefine $\beta$ to avoid direct reference to lines, for which we refer the reader to \cite{Sc07_metric}.
If one removes the assumption of Ahlfors regularity, then one may define $\beta$ in a natural way, and get part (B) of Theorem \ref{t:Hilbert-tst} \cite{Hahlomaa-non-AR}, however, then part (A) of Theorem \ref{t:Hilbert-tst} fails, as is seen by Example 3.3.1 in   \cite{Sc07_survey}.
 
The case in which the metric space is the first Heisenberg group, $\bH$, was studied by Ferrari, Franchi, Pajot as well as Juillet and  Li and the second author. There, the natural definition of $\beta_{\bH}$ uses so-called ``horizontal lines'' in place of Euclidean lines. This was first done in \cite{FFP07}, where the authors proved the Heisenberg group analogue of part (B) of Theorem \ref{t:Hilbert-tst}. However, an example that shows that the analogue of part (A) is false appeared in \cite{Juillet-counter}.  In \cite{LS14} it was shown that the analogue of part (A) does hold if one increases the exponent of $\beta_{\bH}$ from $2$ to $4$, and in \cite{LS15} it was shown that part (B) holds for any exponent $p<4$. (Note that by the definition of the Jones numbers, $\beta_{\bH} \leq 1$, so increasing the power reduces the quantity on the right-hand side of part (B).)

 The metric spaces that we consider in this paper are constructed as limits of metric graphs satisfying certain axioms. These spaces have their origin in a construction of Laakso \cite{La00}, which was later re-interpreted and significantly generalized by Cheeger and Kleiner in \cite{CK13_PI}. (See also \cite{LP01}, Theorem 2.3.) The interest in these constructions is that they provide spaces that possess many strong analytic properties in common with Euclidean spaces, while being geometrically highly non-Euclidean.

To illustrate this principle, we note that the metric measure spaces constructed in \cite{CK13_PI} are doubling and support Poincar\'e inequalities in the sense of \cite{HK98}. Hence, they possess many rectifiable curves, are well-behaved from the point of view of quasiconformal mappings, and by a theorem of Cheeger \cite{Ch99}, they admit a form of first order differential calculus for Lipschitz functions (in particular, a version of Rademacher's theorem) which has been widely studied (e.g. in \cite{Ke04}, \cite{Ba15}, \cite{CKS15}). The examples of \cite{CK13_PI} also admit bi-Lipschitz embeddings into the Banach space $L_1$.

On the other hand, these spaces are highly non-Euclidean in many geometric ways: they have topological dimension one but may have arbitrary Hausdorff dimension $Q\geq 1$, and, under certain mild non-degeneracy assumptions, have no manifold points and admit no bi-Lipschitz embeddings into any Banach space with the Radon-Nikodym property, including Hilbert space. (Other infinitesimal properties of these spaces are considered in Section 9 of \cite{CKS15}.)

Our focus will be on the metric properties of these spaces and so we drop any mention of a measure on these spaces from now on. We will show that, properly viewed, these spaces support a strong form of the Analyst's traveling salesman theorem described above.

\subsection{Definition of the spaces and notation}
Our general metric notation is fairly standard. If $(X,d)$ is a metric space, $x\in X$, and $r>0$, then we write
$$ B(x,r) = \{ y\in X: d(x,y)<r\}.$$
If we wish to emphasize the space $X$, we may write this as $B_X(x,r)$. If $B= B(x,r)$ is a ball and $\lambda>0$, we write $\lambda B$ for $B(x,\lambda r)$. Finally, if $E$ is a set in $X$ and $\delta>0$, we write
$$ N_\delta(E) = \{x\in X: \dist(x, E) < \delta\}.$$

Our spaces will be inverse limits of connected simplicial metric graphs. The inverse systems will be of the form
$$ X_0 \xleftarrow{\pi_0} X_1 \xleftarrow{\pi_1} \dots \xleftarrow{\pi_{i-1}} X_i \xleftarrow{\pi_i} \dots ,$$
where $\{X_i\}$ are metric graphs and $\pi_i$ are mappings that satisfy a few axioms. These are exactly axioms 1-3 from \cite{CK13_PI} (those that concern the metric and not the measure), with the added assumption that $X_0$ is isometric to $[0,1]$. 

\begin{definition}\label{admissiblesystem}
An inverse system is \textit{admissible} if, for some constants $\eta>0$, $2\leq m\in\mathbb{N}$, $\Delta>0$ and for each $i\in \mathbb{N}\cup\{0\}$, the following hold:
\begin{enumerate}[(1)]
\item\label{admsys1} $(X_i,d_i)$ is a nonempty connected graph with all vertices of valence at most $\Delta$, and such that every edge of $X_i$ is isometric to an interval of length $m^{-i}$ with respect to the path metric $d_i$.
\item\label{admsys2} $(X_0,d_0)$ is isometric to $[0,1]$
\item\label{admsys3} If $X'_i$ denotes the graph obtained by subdividing each edge of $X_i$ into $m$ edges of length $m^{-(i+1)}$, then $\pi_i$ induces a map $\pi_i : (X_{i+1}, d_{i+1}) \rightarrow (X'_i , d_i)$ which is open, simplicial, and an isometry on every edge.
\item\label{admsys4} For every $x_i \in X'_i$, the inverse image $\pi_i^{-1} (x_i) \subset X_{i+1}$ has $d_{i+1}$-diameter at most $\eta m^{-i}$.
\end{enumerate}
The constants $m, \eta, \Delta$ are called the \textit{data} of $X$.
\end{definition}

\begin{rmk}
The upper diameter bound $\eta m^{-i}$ for point-preimages under $\pi_i$ in \eqref{admsys4} is written as $\theta m^{-(i+1)}$ in \cite{CK13_PI} for some constant $\theta$; this cosmetic change is convenient in Part \ref{constructionpart} of the present paper.
\end{rmk}

The spaces $X$ we will consider will be inverse limits of admissible inverse systems. Namely, given an inverse system as in Definition \ref{admissiblesystem}, let
$$ X  = \{ (v_i) \in \prod_{i=0}^\infty X_i : \pi_i(v_{i+1}) = v_i \text{ for all } i\geq 0\}.$$
For any $(v_i), (w_i)\in X$, the sequence $d_i(v_i, w_i)$ is increasing (as $\pi_i$ are all $1$-Lipschitz) and bounded above (by Lemma \ref{piproperties} below). We can therefore define the metric on $X$ by
$$ d((v_i),(w_i)) = \lim_{i\rightarrow\infty} d_i(v_i,w_i).$$
The space $(X,d)$ is also isometric to the Gromov-Hausdorff limit of the metric spaces $(X_i,d_i)$ (as follows from Lemma \ref{piproperties2} below; see also Proposition 2.17 of \cite{CK13_PI}).

For each $i\geq 0$, there is also a $1$-Lipschitz mapping
$$\pi^\infty_i:X\rightarrow X_i \cong [0,1]$$
which sends $(v_j)_{j=0}^\infty$ to $v_i$. These mappings factor through the mappings $\pi_i$ in a natural way, i.e.,
$$\pi^\infty_i = \pi_i \circ \pi^\infty_{i+1}.$$ 

\begin{rmk}
To avoid cumbersome subscripts, we now make the following notational choice: For any $j> i\in\mathbb{N}\cup\{0\}$, we denote the mapping from $X_j$ to $X_i$ (determined by composing $\pi_i \circ \pi_{i+1} \circ \dots \circ \pi_{j-1}$) also by $\pi_i:X_j \rightarrow X_i$. Similarly, we generally denote the induced projection $\pi^\infty_i$ from $X$ to $X_i$ by $\pi_i$.

In other words, mappings are always labeled by their range, and the domain will be clear from context. If we wish to emphasize the domain, we will write $\pi^j_i:X_j\rightarrow X_i$ or $\pi^\infty_i:X\rightarrow X_i$.
\end{rmk}

\subsection{Main results}\label{mainresults}

For each integer $n\geq 0$ and each edge $e_n$ in $X_n$, let $p\in X$ be any preimage under $\pi_n$ of the center point of $e_n$. By Lemma \ref{piproperties2} below, for each $n\geq 0$ the (finite) collection of all such $p$ form a $m^{-n}$-separated set in $X$ that is $m^{-n}$-dense in $X$. (In other words, any two such points $p$ are distance at least $m^{-n}$ apart, and each point of $X$ is within distance at most $m^{-n}$ of some such point $p$.) 

Therefore, for each $n$, balls of the form $B(p,Am^{-n})$, for some integer $A$, form a reasonable $m$-adic collection of balls to measure $\beta$-numbers with respect to, analogous to the balls in Theorem \ref{t:Hilbert-tst} or the triples of dyadic cubes in \cite{Jones-TSP}. From now on, it will suffice to fix $A$  to be the least integer greater than $100 C_\eta$ (where $C_\eta$ appears in Lemma \ref{piproperties2} below).

Let $\mathcal{B}$ denote the collection of such balls, for all $n\geq 0$ and all edge-midpoints $p$. Given an edge $e_n\subset X_n$, we refer to the associated ball in $X$ by $B(e_n)$; keep in mind that $B(e_n)$ is a ball of radius $Am^{-n}$ in $X$, and not in $X_n$. Observe that for $A\geq C_\eta$, the ball $B(e_n)$ contains the full pre-image of $e_n$ under $\pi_n:X\rightarrow X_n$.

We will prove the following two results. In both of these results, the $\beta$-numbers we use are defined with respect to a distinguished family of geodesics, the so-called ``monotone geodesics'', which here play the role of lines in Hilbert space or horizontal lines in the Heisenberg group. See subsections \ref{furtherprelims} and \ref{betasubsection} for the definitions.

\begin{thm}\label{upperbound}
Let $X$ be the limit of an admissible inverse system (as in Definition \ref{admissiblesystem}).

For every $p>1$, there is a constant $C_{p,X}>0$ with the following property: Whenever $\Gamma\subset X$ is a compact, connected set, we have
$$ \sum_{B\in \mathcal{B}} \beta_\Gamma(B)^p \diam(B) \leq C_{p,X}\mathcal{H}^1(\Gamma). $$
The constant $C_{p,X}$ depends on $p$ and the data of $X$.
\end{thm}
This is sharp in the sense that when $p=1$ the theorem is false; see Section \ref{counterexample}.

\begin{thm}\label{construction}
Let $X$ be the limit of an admissible inverse system (as in Definition \ref{admissiblesystem}).

There are constants $C_X>1$ and $\epsilon>0$, depending only on the data of $X$, with the following property: Let $E\subset X$ be compact. Then there is a compact connected set $G\subset X$ containing $E$ such that
$$ \mathcal{H}^1(G) \leq C_X\left(\diam(E) + \sum_{B\in \mathcal{B}, \beta_E(B)\geq\epsilon} \diam(B)\right). $$
\end{thm}

What these results together show is that, in limits of admissible inverse systems, one may characterize subsets of rectifiable curves as those which are quantitatively close to being monotone at most locations and scales.

\section{Preliminaries}
\subsection{Basic lemmas}\label{furtherprelims}
We now collect some further definitions and facts about limits of admissible inverse systems.

The following basic properties of the maps $\pi_i$ can be found in Section 2.3 of \cite{CK13_PI}.
\begin{lemma}\label{piproperties}
For each $i\in\mathbb{N}$, $x_{i+1}, y_{i+1}\in X_{i+1}$, and $r>0$, the map $\pi_i: X_{i+1}\rightarrow X_i$ has the following properties. 
\begin{enumerate}[(a)]
\item\label{piballsurjective} $\pi_i(B(x_{i+1},r)) = B(\pi_i(x_{i+1}),r)$
\item\label{pipathlift} If $c:[a,b]\rightarrow X_i$ is path parametrized by arc length and $c(a) = \pi(x_{i+1})$, then there is a lift $\tilde{c}:[a,b]\rightarrow X_{i+1}$ of $c$ under $\pi_i$, parametrized by arc length, such that $\tilde{c}(a) = x_{i+1}$. (That $\tilde{c}$ is a lift of $c$ under $\pi_i$ means that $\pi_{i}\circ \tilde{c} = c$.)
\item\label{pipreimage} For each $i\geq 0$, 
$$d_i(\pi_i(x_{i+1}),\pi_i(y_{i+1})) \leq d_{i+1}(x_{i+1}, y_{i+1}) \leq d_i(\pi_i(x_{i+1}),\pi_i(y_{i+1})) + 2\eta m^{-i}.$$ 
\end{enumerate}
\end{lemma}

By iterating these properties, we immediately obtain the following about the mappings $\pi^\infty_i:X\rightarrow X_i$.
\begin{lemma}\label{piproperties2}
There is a constant $C_\eta\geq 1$, depending only on $\eta$, with the following properties. For each $i\geq 0$, $x, y\in X$, and $r>0$, we have:
\begin{enumerate}[(a)]
\item\label{piballsurjective2} $\pi^\infty_i(B(x,r)) = B(\pi^\infty_i(x),r)$.
\item\label{pipathlift2}  If $i\geq 0$, $c:[a,b]\rightarrow X_i$ is a path parametrized by arc length, and $\pi^\infty_i(x)=c(a)$, then there is a path $\tilde{c}:[a,b]\rightarrow X$, parametrized by arc length, such that $\pi^\infty_i \circ \tilde{c} = c$ and $\tilde{c}(a) = x$.
\item\label{pipreimage2} For each $i\geq 0$, 
$$d_i(\pi^\infty_i(x),\pi^\infty_i(y)) \leq d(x, y) \leq d_i(\pi^\infty_i(x),\pi^\infty_i(y)) + C_\eta m^{-i},$$
where $C_\eta$ is a constant depending only on $\eta$.
\item\label{piballpreimage2} $(\pi^\infty_i)^{-1}(B(\pi^\infty_i(x),r)) \subseteq B(x, r+C_\eta m^{-i})$.
\end{enumerate}
\end{lemma}

The map $\pi_0:X\rightarrow X_0\cong [0,1]$ has an interesting property: it is ``Lipschitz light'' in the sense of \cite{CK13_inverse}. In particular, we will use the fact that $\pi_0$ preserves the diameter of connected sets up to a fixed multiplicative constant, as we briefly prove below. (For a much more general theorem which implies the next lemma, see Theorem 1.10 of \cite{CK13_inverse}.)

\begin{lemma}
Let $X$ be the limit of an admissible inverse system. There is a constant $C$, depending only on the data of $X$, such that if $E\subset X$ is connected, then
\begin{equation}\label{pipreservesdiam}
\diam(\pi_0(E)) \geq \frac{1}{C}\diam(E).
\end{equation}
\end{lemma}
\begin{proof}
Choose $n\in\mathbb{N}$ such that
$$m^{-(n+1)}\leq \diam(\pi_0(E)) <m^{-n}.$$
It follows that $\pi_n(E)$ is contained in $B_{X_n}(v_n, m^{-n})$, i.e., the ``star'' of a single vertex $v_n$ in $X_n$. Indeed, if not, then the connected set $\pi_n(E)$ would contain two distinct vertices of $X_n$, and so $\pi_0(E)$ would contain two distinct points of $m^{-n}\mathbb{Z}$. But this is impossible as $\diam(\pi_0(E))<m^{-n}$.

It then follows from Lemma \ref{piproperties2} that
$$ \diam(E) \leq \diam(\pi_n(E)) + C_\eta m^{-n} \leq 2m^{-n} + C_\eta m^{-n} \leq C\diam(\pi_0(E)).$$
\end{proof}

Recall that a metric space is \textit{doubling} if there is a constant $N$ such that every ball in the space can be covered by $N$ balls of half the radius. A metric space is \textit{geodesic} if every two points can be joined by an arc whose length is equal to the distance between the two points; such an arc is called a \textit{geodesic}.

The following lemma is similar to Lemma 3.2 of \cite{CK13_PI}, but since here we have no measures we argue in a slightly different way.

\begin{lemma}
The limit $X$ of an admissible inverse system is doubling and geodesic. The doubling constant $N$ depends only on the data of $X$.
\end{lemma}
\begin{proof}
That $X$ is geodesic follows from the fact that it is a Gromov-Hausdorff limit of geodesic spaces.

To show $X$ is doubling, it suffices to show that the spaces $(X_i, d_i)$ are uniformly doubling. 

Let us first make the following observation: If $r \leq K m^{-j}$ and $x_j \in X_{j}$, then $B =B(x_j,r)$ in $X_j$ is covered by no more than $N_K$ balls of radius $r/2$, independent of $j$. Indeed, $B$ is a subset of a finite simplicial graph with at most $\Delta^K$ vertices, and finite graphs are doubling with constant depending only on the number vertices.

Now consider a ball $B=B(x_{i}, r)\subset X_{i}$. Let $K= \max(\frac{8m^2\eta}{m-1},1)$. Choose $j$ such that
$$ Km^{-(j+1)} < r \leq Km^{-j}.$$
If $j\geq i$, then the ball $B$ is covered by at most $N_K$ balls of radius $r/2$, by the observation above.

If $j<i$, consider the ball
$$ B_{X_j}(x_j,r) = \pi_j(B) \subset X_j $$
where $x_j=\pi_j(x_{i}) \in X_j$. (See Lemma \ref{piproperties}.)

Then $\pi_j(B)$ is covered by $N_K^2$ balls $\{B^j_\ell\}_{l=1}^{N_K^2}$ of radius $r/4$ in $X_j$. Let $B^{i}_\ell$ denote the preimage of $B^j_\ell$ under the projection from $X_{i}$ to $X_j$; these preimages cover $B$ in $X_{i}$. By Lemma \ref{piproperties}, each $B^{i}_\ell$ is contained in a ball of radius
$$ r/4 + 2\eta(m^{-j} + \dots + m^{-(i+1)}) \leq r/2.$$
So $B$ is covered by at most $N_K^2$ balls of radius $r/2$ in $X_{i+1}$. This completes the proof that the graphs $X_i$ are all uniformly doubling, and therefore that $X$ is doubling.
\end{proof}

The following definition describes the sub-class of geodesics that we will consider. In our spaces, they will play the role that lines play in Hilbert space or that horizontal lines play in the Heisenberg group.

\begin{definition}
Let $X$ be the limit of an admissible inverse system. A geodesic arc $\gamma$ in $X$ or in any $X_n$ is called a \textit{monotone geodesic} if the restriction $\pi_0|_\gamma$ is an isometry onto $X_0 \cong [0,1]$, i.e., if
$$ |\pi_0(x) - \pi_0(y)| = d(x,y) \text{ for all } x,y\in \gamma$$
A \textit{monotone geodesic segment} is a connected subset of a monotone geodesic.
\end{definition}

In the next lemma, we collect some basic observations about monotone geodesic segments. A \textit{path} in a metric space is a continuous image of $[0,1]$.

\begin{lemma}\label{monotonefacts}
Let $X$ be the limit of an admissible inverse system. Let $\gamma$ be a path in $X$ or in $X_n$ for some $n\geq 0$.
\begin{enumerate}[(a)]
\item\label{monotonefacts1} If $\gamma$ is a monotone geodesic segment, then $\pi_i(X)$ is a monotone geodesic segment for each $i\geq 0$ (with $i\leq n$ if $\gamma\subset X_n$).
\item\label{monotonefacts2} If $\gamma$ is a monotone geodesic segment, then $\gamma$ can be extended (not necessarily uniquely) to a monotone geodesic.
\item\label{monotonefacts3} The path $\gamma$ is a monotone geodesic segment if and only if the restriction $\pi_0|_\gamma$ is injective.
\end{enumerate}
\end{lemma}
\begin{proof}
Part \eqref{monotonefacts1} follows immediately from the definition of monotone geodesic and the fact that all projections $\pi_i$ are $1$-Lipschitz.

Part \eqref{monotonefacts2} follows from part \eqref{pipathlift2} of Lemma \ref{piproperties2}.

Finally, we address part \eqref{monotonefacts3}. We need to show that if $\pi_0|_\gamma$ is injective, then $\pi_0|_\gamma$ is an isometry. By part \eqref{pipathlift2} of Lemma \ref{piproperties2}, we may assume without loss of generality that $\pi_0(\gamma) = X_0 \cong [0,1]$. Suppose first that $\gamma\subset X_n$ for some $n\geq 0$.

Since $\pi_0|_\gamma$ is an injective continuous function on a compact set, it has a continuous inverse $\alpha:[0,1]\rightarrow \gamma \subset X$, which parametrizes $\gamma$. As $\pi_0$ is an isometry on each edge, $\alpha$ is $1$-Lipschitz. Hence, for any $x=\alpha(s),y=\alpha(t)\in \gamma$, we have
$$ d(x,y) = d(\alpha(s),\alpha(t)) \leq |s-t| = |\pi_0(\alpha(s)) - \pi_0(\alpha(y))| = |\pi_0(x)-\pi_0(y)| \leq d(x,y), $$
and hence $\pi_0|_\gamma$ is an isometry.

Now, if $\gamma$ is a path in $X$ and $\pi_0|_\gamma$ is injective, then consider the paths $\gamma_n=\pi_n(\gamma)\subset X_n$ for all $n\geq 0$. The restriction $\pi_0|_{\gamma_n}$ is injective for each $n$, and hence an isometry by the above calculation. It follows that the path $\gamma$ in $X$ is a monotone geodesic segment.
\end{proof}

For a natural pictorial example of a monotone geodesic in a specific admissible inverse system, see Figure \ref{fig:monotonediamond} in Example \ref{langplaut} below.

Finally, the following definition will be convenient at various points below.

\begin{definition}\label{levelvertices}
A point $x$ of $X$ or any $X_n$ is said to be a \textit{vertex of level $i$}, for some natural number $i\geq 0$, if $\pi_0(x) \in m^{-i}\mathbb{Z}$. 
\end{definition}
In this language, the vertices of the graph $X_n$ are the vertices of level $n$ in $X_n$. If $n\leq i$, the vertices of level $i$ in $X_n$ are exactly the images in $X_n$ of the vertices of the graph $X_i$ under the projection map $\pi_n:X_i\rightarrow X_n$. Of course, if $i<j$, then a vertex of level $i$ is also a vertex of level $j$.

\subsection{Examples}
We now describe two (already well-known) examples of admissible inverse systems.

\begin{example}[\cite{LP01}, see also \cite{CK13_inverse}, Example 1.2]\label{langplaut}
In this example, the scale factor $m$ will be $4$. As required by the axioms, let $X_0$ be a graph with two vertices and one edge of length $1$.

Let $G$ denote the graph with six vertices and six edges depicted in Figure \ref{fig:langplautG}.
\begin{figure}
	\centering
		\includegraphics[scale=0.2]{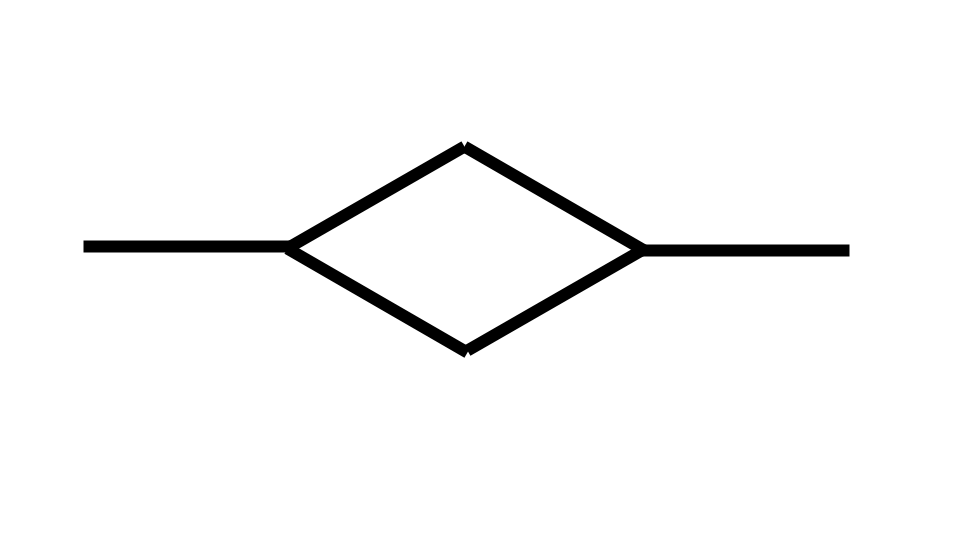}
		\caption{The graph $G$ in Example \ref{langplaut}}
	\label{fig:langplautG}
\end{figure}

For each $i\geq 1$, $X_i$ is formed by replacing each edge of $X_{i-1}$ with a copy of the graph $G$ in which each edge has length $m^{-i}$. The resulting graph is endowed with the shortest path metric. The projection map $\pi_{i-1}:X_i \rightarrow X_{i-1}$ is then defined by collapsing each copy of $G$ onto the edge of $X_{i-1}$ that it replaced.

This yields an inverse system
$$ X_0 \xleftarrow{\pi_0} X_1 \xleftarrow{\pi_1} \dots \xleftarrow{\pi_{i-1}} X_i \xleftarrow{\pi_i} \dots $$
and it is easy to see that this system satisfies the requirements of Definition \ref{admissiblesystem}.

Stages $X_1$ and $X_2$ in this construction are depicted in Figure \ref{fig:diamond}, and in figure \ref{fig:monotonediamond} we illustrate an example of a monotone geodesic in this space.

\begin{figure}
\centering
\begin{subfigure}{.5\textwidth}
  \centering
  \includegraphics[width=.4\linewidth]{diamond1}
  \caption{The graph $X_1$ from Example \ref{langplaut}.}
  \label{fig:diamond1}
\end{subfigure}%
\begin{subfigure}{.5\textwidth}
  \centering
  \includegraphics[width=.4\linewidth]{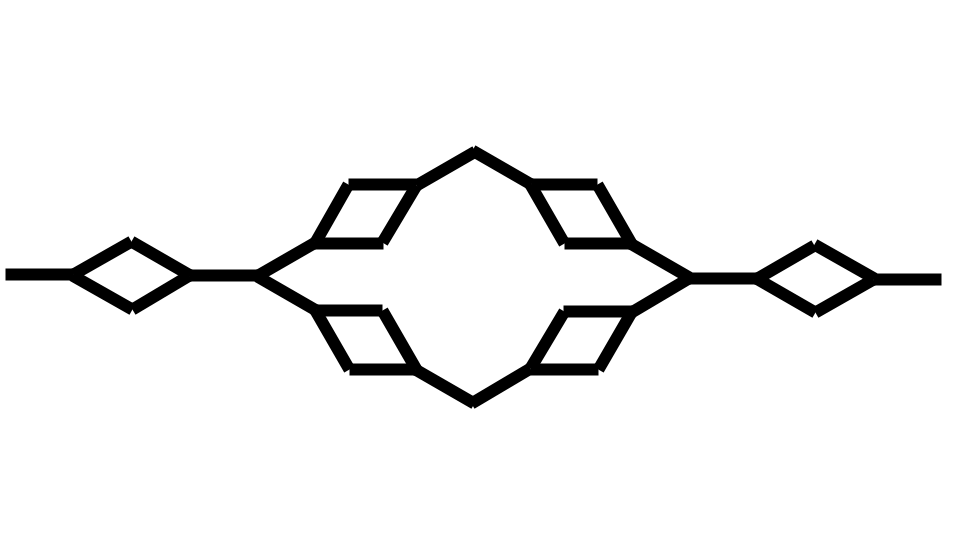}
  \caption{The graph $X_2$ from Example \ref{langplaut}.}
  \label{fig:diamond2}
\end{subfigure}
\caption{Graphs $X_1$ and $X_2$ in Example \ref{langplaut}.}
\label{fig:diamond}
\end{figure}

\begin{figure}
	\centering
		\includegraphics[scale=0.2]{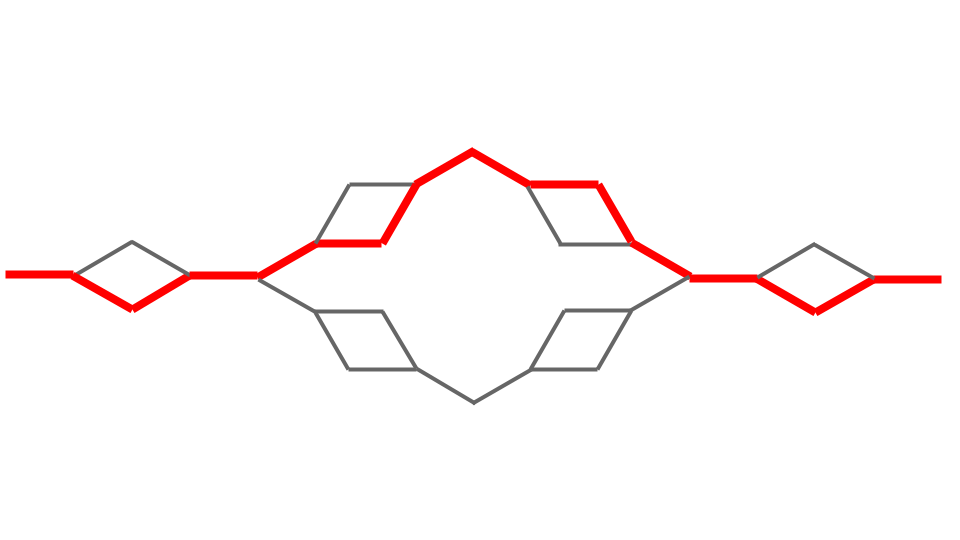}
		\caption{A monotone geodesic in Example \ref{langplaut}}
	\label{fig:monotonediamond}
\end{figure}

\end{example}

The second example we describe is a modification of a construction due to Laakso \cite{La00}. This interpretation of Laakso's' construction in terms of inverse limits of graphs, can be found in \cite{CK13_inverse}, Example 1.4 (with the only modification being that ours is ``dyadic'' rather than ``triadic'').

\begin{example}[\cite{La00}, \cite{CK13_inverse}]\label{dyadiclaakso}
In this example, the scale factor $m$ will be $2$. Let $X_0$ be a graph with two vertices and one edge of length $1$.

For each $i\geq 1$, $X_i$ is formed from $X_{i-1}$ as follows. Let $X'_{i-1}$ be the graph obtained by inserting an extra vertex in the middle of each edge of $X_{i-1}$. Let $B\subset X'_{i-1}$ denote the set of these bisecting vertices. Form $X_i$ by taking two copies of $X'_{i-1}$ and gluing them at the set $B$.

In other words,
$$ X_i \cong \left(X'_{i-1} \times \{0,1\} \right)/ \sim $$
where $(v,0)\sim (v,1)$ for all $v\in B$.
	
The resulting graph $X_i$ is endowed with the shortest path metric.	
	
The map $\pi_{i-1}:X_i\rightarrow X_{i-1}$ is formed by collapsing the two copies of $X_{i-1}$ into one; i.e., it is induced by the map from $X'_{i-1} \times \{0,1\}$ to $X_{i-1}$ that sends $(v,j)$ to $v$.

It is again easy to verify that this inverse system satisfies the requirements of Definition \ref{admissiblesystem}. Stages $X_1$ and $X_2$ in this construction are depicted in Figure \ref{fig:laakso}.

\begin{figure}
\centering
\begin{subfigure}{.5\textwidth}
  \centering
  \includegraphics[width=.6\linewidth]{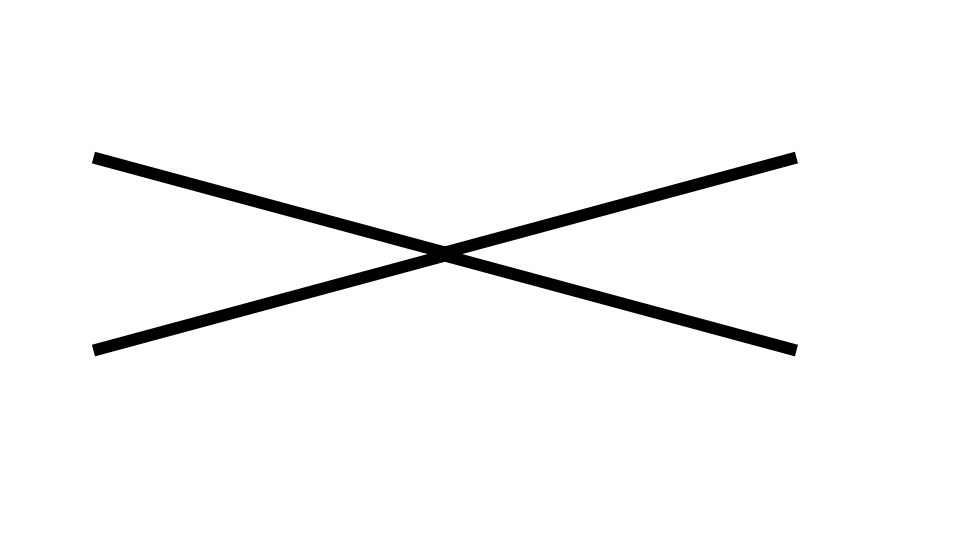}
  \caption{The graph $X_1$ from Example \ref{dyadiclaakso}.}
  \label{fig:laakso1}
\end{subfigure}%
\begin{subfigure}{.5\textwidth}
  \centering
  \includegraphics[width=.6\linewidth]{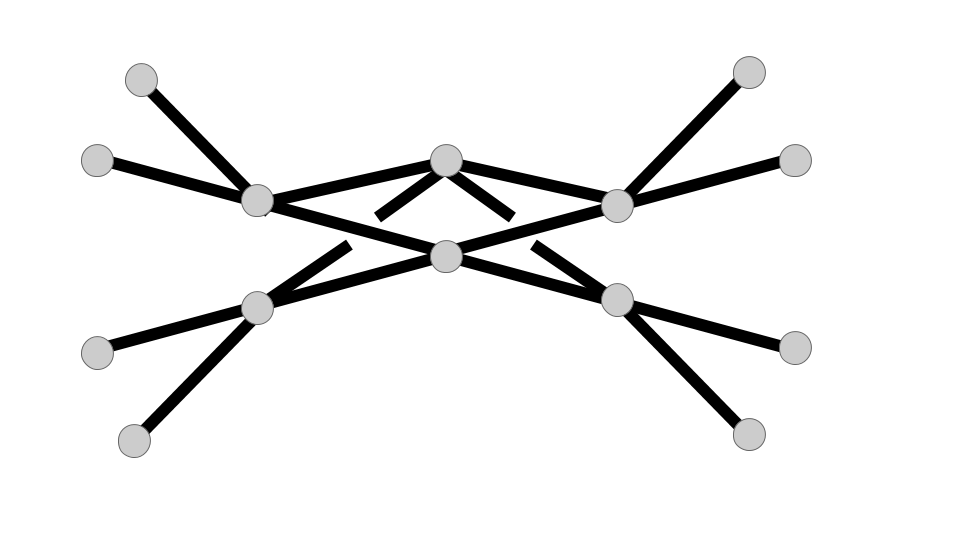}
  \caption{The graph $X_2$ from Example \ref{dyadiclaakso}.}
  \label{fig:laakso2}
\end{subfigure}
\caption{Graphs $X_1$ and $X_2$ in Example \ref{dyadiclaakso}.}
\label{fig:laakso}
\end{figure}
\end{example}

Note that, while both Examples \ref{langplaut} and \ref{dyadiclaakso} are highly self-similar, this is not a general requirement for admissible inverse systems. 

\subsection{Beta numbers}\label{betasubsection}
For a set $E\subset X$ and a metric ball $B\subset X$, we define
$$ \beta_E(B) = \frac{1}{\diam(B)}\inf_{L} \sup_{x\in E\cap B} \dist(x,L),$$
where the infimum is taken over all monotone geodesics $L$.

For a function $f:[0,1]\rightarrow \RR$ and an arc $\tau=f|[a,b]$, we follow \cite{Sc07_hilbert} and define 
$$\tilde{\beta}_f(\tau) = \frac{1}{\diam(\tau)} \inf_{t\in [a,b]} \dist(f(t),[f(a),f(b)]).$$
Here $[f(a),f(b)]$ is shorthand for the closed interval $[\min(f(a),f(b)), \max(f(a),f(b))]$.

The definition of $\tilde{\beta}_f$ makes perfect sense for mappings $f$ into higher dimensional Euclidean spaces, but here we will only need it for mappings into $\RR$.

\part{The upper bound}\label{upperboundpart}

In part \ref{upperboundpart}, we prove Theorem \ref{upperbound}. For the remainder of this part, we fix an admissible inverse system as in Definition \ref{admissiblesystem}, with limit $X$.

\section{The curve, the balls, and a reduction}
Fix $\Gamma\subseteq X$ compact and connected with $\mathcal{H}^1(\Gamma)<\infty$. 

Let $K_1 \supset K_2 \supset K_3 \supset ...$ be a family of nested subsets of $\Gamma$ such that $K_i$ is a maximal $m^{-i}$-separated set in $\Gamma$ for each $i\geq 0$. For a fixed constant $A>0$, let
$$\mathcal{G} = \{B(x,2Am^{-n}): n\in\mathbb{N}, x\in K_n\}$$ 

We first observe that, in order to prove Theorem \ref{upperbound}, we can instead sum over the balls of $\mathcal{G}$.

\begin{lemma}\label{reductionlemma}
There is a constant $C_{p,X}$, depending only on $p$ and the data of $X$, with the following property: For any compact, connected set $\Gamma\subset X$, we have
$$ \sum_{B\in\mathcal{B}} \beta_\Gamma(B)^p \diam(B) \leq C_{p,X} \sum_{B'\in\mathcal{G}} \beta_\Gamma(B')^p \diam(B')$$
\end{lemma}
\begin{proof}
Observe that if a ball $B\in\mathcal{B}$ has non-empty intersection with $\Gamma$, then $B$ is contained in a ball $B'\in\mathcal{G}$ such that
\begin{equation}\label{samediameter}
\diam(B) \leq \diam(B') \leq 4\diam(B).
\end{equation}
Furthermore, because $X$ is a doubling metric space, there are at most $C_X$ balls $B'\in\mathcal{G}$ containing $B$ and satisfying \eqref{samediameter}, where $C_X$ depends only on the data of $X$.

In addition, if $B'$ contains $B$ and satisfies \eqref{samediameter}, then
$$ \beta_\Gamma(B)\leq 4\beta_\Gamma(B').$$

We then have
\begin{align*}
\sum_{B\in\mathcal{B}} \beta_\Gamma(B)^p \diam(B) &\leq \sum_{B'\in\mathcal{G}} \sum_{B\in\mathcal{B}, B\subset B', \eqref{samediameter}} \beta_\Gamma(B)^p \diam(B)\\
&\leq  C_{X,p} \sum_{B'\in\mathcal{G}} \beta_\Gamma(B')^p \diam(B')
\end{align*}
\end{proof}

Therefore, in order to prove Theorem \ref{upperbound}, it suffices to prove that
\begin{equation}\label{upperboundreduction}
\sum_{B\in\mathcal{G}} \beta_\Gamma(B)^p \diam(B) \leq C_{X,p} \mathcal{H}^1(\Gamma)
\end{equation}
where $C_{X,p}$ depends only on $p>1$ and the data of $X$.

Assume without loss of generality that $\mathcal{H}^1(\Gamma)<\infty$. Let $\gamma:[0,1]\rightarrow \Gamma \subset X$ be a parametrization proportional to arc length. (See, e.g., Lemma 2.14 of \cite{LS14} for the existence of this parametrization.) In particular, $\gamma$ is Lipschitz with constant at most $32\mathcal{H}^(\Gamma)$.

Given a ball $B$, we will write
$$ \Lambda(B) = \{\text{connected components of } \gamma^{-1}(2B) \text{ touching } B\}$$

We make one observation immediately, namely that we may focus only on those balls $B\in \mathcal{G}$ that are small and far from the two endpoints $\gamma(0)$ and $\gamma(1)$. In particular, for such balls every arc of $\Lambda(B)$ has both endpoints in $\partial{2B}$. Let  
$$\mathcal{G}_0 = \{B\in\mathcal{G} : \gamma(0)\notin 10m B, \gamma(1)\notin 10m B\}.$$ 

\begin{lemma}\label{G0lemma}
$$\sum_{\mathcal{G}\setminus\mathcal{G}_0} \beta_\Gamma(B)^p \diam(B) \leq C_{m,p} \mathcal{H}^1(\Gamma)$$
\end{lemma}
\begin{proof}
The only way that a ball $B\in\mathcal{G}$ can fail to be in $\mathcal{G}_0$ is if $10m B$ contains $\gamma(0)$ or $\gamma(1)$. 

So it suffices to show the following: If $z$ is a point of $\Gamma$, then
$$ \sum_{B\in\mathcal{G}, z\in 10m B} \beta_\Gamma(B)^p \diam(B) \leq C_p \diam(\Gamma). $$

We write
\begin{align*}
\sum_{B\in\mathcal{G}, z\in 10m B} \beta_\Gamma(B)^p\diam(B) &= \sum_{B\in\mathcal{G}, z\in 10m B, B\supseteq \Gamma} \beta_\Gamma(B)^p\diam(B) + \sum_{B\in\mathcal{G}, z\in 10m B, B\not\supseteq \Gamma} \beta_\Gamma(B)^p\diam(B)\\
&\leq \sum_{B\in\mathcal{G}, z\in 10m B, B\supseteq \Gamma} \diam(\Gamma)^{p}\diam(B)^{1-p} + \sum_{B\in\mathcal{G}, z\in 10m B, B\not\supseteq \Gamma} \diam(B)\\
&\leq C_{m,p} \diam(\Gamma)\\
&\leq C_{m,p} \mathcal{H}^1(\Gamma)
\end{align*}
\end{proof}

\section{The filtrations}\label{filtrations}
The following definition is taken from Lemma 3.11 of \cite{Sc07_hilbert} (adapted to be $m$-adic rather than dyadic). Fix a large constant $J\in\mathbb{N}$.

\begin{definition}
A \textit{filtration} is a family $\mathcal{F}=\cup_{n=n_0}^\infty \mathcal{F}_n$ of sub-arcs of $\gamma$ with the following properties.
\begin{enumerate}[(1)]
\item The integer $n_0$ satisfies $m^{-n_0} \leq \diam(\gamma) < m^{-n_0+1}$.
\item If $\tau_{n+1}\in \mathcal{F}_{n+1}$, then there is a unique $\tau_n\in\mathcal{F}_n$ such that $\tau_{n+1}\subset \tau_n$.
\item If $\tau_n\in \mathcal{F}_n$, then $m^{-nJ}\leq \diam(\tau_n) \leq Am^{-nJ+2}$.
\item If $\tau_n, \tau'_n\in\mathcal{F}_n$, then $\tau_n \cap \tau'_n$ must be either empty, a single point, or two points.
\item $\cup_{\tau_0 \in\mathcal{F}_{n_0}} \tau_0 = \cup_{\tau_n \in\mathcal{F}_n} \tau_n$ for all $n\geq 0$.
\end{enumerate} 
\end{definition}

Let $\gamma_0 = \pi_0 \circ \gamma:[0,1]\rightarrow \RR$. We have to set up some appropriate filtrations, based on $\gamma_0$. There is an obvious correspondence between filtrations of $\gamma$ and filtrations of $\gamma_0$.

\begin{rmk}
When we discuss an arc $\tau$ in a curve $\gamma$ or $\gamma_0$, we are referring to the restriction $\gamma|_{[a,b]}$ or $\gamma_0|_{[a,b]}$ of the curve to an interval $[a,b]\subset [0,1]$. When we say that two arcs $\tau$ and $\tau'$ intersect, or that $\tau\subset \tau'$, we are referring to their associated intervals in the domain of $\gamma$ or $\gamma_0$.

On the other hand, when we refer to the diameter of an an arc $\tau$, we are referring to the diameter of its image. If $\tau$ is an arc of $\gamma$, this image is in $X$, whereas if $\tau$ is an arc of $\gamma_0$, this image is in $\RR$. Note, however, that by \eqref{pipreservesdiam} the diameter of $\tau$ is the same, up to an absolute multiplicative constant, whether it is considered as an arc of $\gamma$ in $X$ or an arc of $\gamma_0$ in $\RR$.

See \cite{LS14}, Remark 2.15 for a similar remark in the Heisenberg group context.
\end{rmk}

For $k\in\mathbb{N}\cup\{0\}$, let $D_k = m^{-k}\mathbb{Z}\subset \RR$ denote the $m$-adic grid at scale $k$ in $\RR$.

Let $E_k = \gamma_0^{-1}(D_k) \subset [0,1]$, and let
$$ F_k = \{t\in E_k : \gamma_0(\sup\{s\in E_k: s<t\}) \neq \gamma_0(t)\} $$
The set $F_k$ contains a time $t$ whenever $\gamma_0(t)$ is an element of $D_k$ different from the previous one. (We include in $F_k$ the first time $t\in E_k$.)

Write $F_k = \{t_1 < t_2 < \dots < t_N\}\subset [0,1]$. It has the following simple properties:
\begin{itemize}
\item $\gamma_0(F_k) = \gamma_0([0,1]) \cap D_k$. 
\item $|\gamma_0(t_i) - \gamma_0(t_{i+1})| = m^{-k}$ for each $i$. (So in particular, as $\gamma_0$ is Lipschitz, $F_k$ is finite.)
\item Each of the sets $\gamma_0([0,t_1])$, $\gamma_0([t_i, t_i+1])$, and $\gamma_0([t_N, 1])$ has diameter $<2m^{-k}$.
\end{itemize}

The following lemma will be useful later on.

\begin{lemma}\label{Fkdense}
Let $[a,b]\subset [0,1]$ be an interval such that $\diam\left(\gamma_0([a,b])\right) \geq 2m^{-k}$. Then
\begin{equation}\label{Fkdense1}
\gamma_0([a,b]) \subset N_{2m^{-k}}(\gamma_0([a,b]\cap F_k)) \subset \RR
\end{equation}
and
\begin{equation}\label{Fkdense2}
\gamma([a,b]) \subset N_{2Cm^{-k}}(\gamma([a,b]\cap F_k)) \subset X.
\end{equation}
\end{lemma}
\begin{proof}
Take $s\in [a,b]$. Choose consecutive points $t_i < t_{i+1}$ in $F_k \cup \{0,1\}$ such that $s\in [t_i, t_{i+1})$. Because $\diam \left(\gamma_0([a,b])\right) \geq 2m^{-k}$, it must be the case that either $t_i$ or $t_{i+1}$ is in $[a,b]$. (Otherwise $\diam \left(\gamma_0([a,b])\right) \leq \diam \left(\gamma_0[t_i, t_{i+1}]\right) < 2m^{-k}$.)

It then follows that
$$ \dist(\gamma_0(s) , \gamma_0([a,b]\cap F_k)) \leq \diam \left(\gamma_0([t_i, t_{i+1}])\right) < 2m^{-k},$$
which proves (\ref{Fkdense1}).

The second part, (\ref{Fkdense2}), follows from \eqref{pipreservesdiam}. Indeed,
$$ \dist(\gamma(s) , \gamma([a,b]\cap F_k)) \leq \diam \left(\gamma([t_i, t_{i+1}])\right) \leq C\diam \left(\gamma_0([t_i, t_{i+1}])\right) \leq 2Cm^{-k}.$$
\end{proof}

We now define, for each $k\in\mathbb{N}$, a collection of arcs
$$\mathcal{F}^0_k  = \{ \gamma_0|_{[t_i, t_{i+2}]} : t_i<t_{i+1}<t_{i+2} \text{ are consecutive in } F_k\}.$$
Then the collections $\mathcal{F}^0_k$ satisfy the conditions of Lemma 3.13 of \cite{Sc07_hilbert}, with $C=4$ and $A=1$ (and dyadic scales replaced by $m$-adic scales). So (by a slight modification of Lemma 3.13 of \cite{Sc07_hilbert}), we obtain the following:

\begin{lemma}\label{filtration}
For some constant $J=J_X>0$, there are $J$ filtrations $\mathcal{F}^1, \dots, \mathcal{F}^{CJ}$ of $\gamma$ with the following property:

If $\tau = \gamma_0|[a, b] \subset \mathcal{F}^0_k$, then there is an arc $\tau'=\gamma_0|[a',b']$ in a filtration $\mathcal{F}^i$ such that $\tau'\supseteq \tau$ and 
$$ \max\{\diam \left(\gamma_0([a',a])\right), \diam \left( \gamma_0([b,b']) \right)\} \leq \frac{1}{10}\diam(\tau).$$
\end{lemma}

For convenience, let $\hat{\mathcal{F}}$ denote the union $\mathcal{F}^1 \cup \dots \cup \mathcal{F}^{J}$. Let $C'=C'_X=5Cm$, where $C$ is the constant from \eqref{pipreservesdiam}. Given a ball $B$, let $\hat{\mathcal{F}}(B)$ denote the collection of all arcs $\tau = \gamma_0|[a,b]$ in $\hat{\mathcal{F}}$ such that $\gamma([a,b]) \subseteq 4C' B$.

We now divide our collection $\mathcal{G}$ of balls into two types. Fix an absolute constant $\epsilon_0>0$, which will be chosen to be sufficiently small depending on the data of $X$. Recall the definitions of $\beta$ and $\tilde{\beta}$ from subsection \ref{betasubsection}.

$$ \mathcal{G}_1 = \{B\in\mathcal{G}_0 : \exists \tau \in \hat{\mathcal{F}}(B) \text{ with } \tilde{\beta}_{\gamma_0}(\tau) \geq 1/10 \text{ and } \diam (\tau) \geq \epsilon_0 \beta_\Gamma(B) \diam(B)\}$$

$$ \mathcal{G}_2 = \mathcal{G}_0\setminus\mathcal{G}_1 = \{B\in\mathcal{G}_0 : \tilde{\beta}_{\gamma_0}(\tau) < 1/10 \text{ for every arc } \tau\in \hat{\mathcal{F}}(B) \text{ with } \diam(\tau) \geq \epsilon_0 \beta_\Gamma(B) \diam(B)\}$$

Balls in $\mathcal{G}_1$ are ``non-flat'' (have one large, non-flat arc), while balls in  $\mathcal{G}_2$ are ``flat'' (i.e., consist of multiple far flat pieces).  The idea of dividing the collection of balls into two sub-collections with these qualitative features is now standard (see \cite{Ok92}, \cite{Sc07_hilbert}, \cite{Sc07_metric}, \cite{LS14}) but our particular division is sensitive to the current context.

\section{Non-flat balls}

\begin{prop}\label{nonflat}
For any $p>1$, we have 
$$\sum_{B\in \mathcal{G}_1} \beta_\Gamma(B)^p \diam(B) \leq C_{p,X} \mathcal{H}^1(\Gamma).$$
\end{prop}
\begin{proof}
First note that if $E$ is a set in $X$ with at least one point and $q>0$, then
\begin{equation}\label{doublingbound}
\sum_{B \in \mathcal{G}, 4C'B\supseteq E} \diam(B)^{-q} \leq C_{q,X} \diam(E)^{-q}
\end{equation}
because of the doubling property of $X$. (At most $C_X$ balls of each scale bigger than $\diam(E)$ can contain $E$.)

Now, for each $B$ in $\mathcal{G}_1$, let $\tau_B\subset 4C'B$ be an arc in $\hat{\mathcal{F}}(B)$ such that
$$  \tilde{\beta}_{\gamma_0}(\tau_B) \geq 1/10 \text{ and } \diam(\tau_B) \geq \epsilon_0 \beta_\Gamma(B) \diam(B). $$
Then, using (\ref{doublingbound}), we have
\begin{align*}
\sum_{B\in \mathcal{G}_1} \beta_\Gamma(B)^p \diam(B) &= \sum_{B\in \mathcal{G}_1} \left(\beta_\Gamma(B) \diam(B)\right)^p \diam(B)^{1-p}\\
&\leq C_{p,X} \sum_{B\in \mathcal{G}_1} \diam (\tau_B)^p \diam(B)^{1-p}\\
&= C_{p,X} \sum_{\tau \in \hat{\mathcal{F}}, \tilde{\beta}_{\gamma_0}(\tau)\geq 1/10} \diam(\tau)^p \sum_{B\in\mathcal{G}_1, \tau_B = \tau} \diam(B)^{1-p}\\
&\leq C_{p,X} \sum_{\tau\in \hat{\mathcal{F}}, \tilde{\beta}_{\gamma_0}(\tau)\geq 1/10} \diam(\tau)^p \diam(\tau)^{1-p}\\
&= C_{p,X} \sum_{\tau\in \hat{\mathcal{F}}, \tilde{\beta}_{\gamma_0}(\tau)\geq 1/10} \diam(\tau)\\
&\leq C_{p,X} \sum_{\tau\in \hat{\mathcal{F}}, \tilde{\beta}_{\gamma_0}(\tau)\geq 1/10} \tilde{\beta}_{\gamma_0}(\tau)^2 \diam(\tau)\\
&= C_{p,X} \sum_{\tau\in \hat{\mathcal{F}}} \tilde{\beta}_{\gamma_0}(\tau)^2 \diam(\tau)\\
&\leq C_{p,X} \mathcal{H}^1(\Gamma)
\end{align*}
The last inequality follows from Lemma 3.11 of \cite{Sc07_hilbert} (in the case where the image is $\RR$), and the fact that there are a controlled number $J=J_X$ of different filtrations making up $\hat{\mathcal{F}}$. (Observe that the proof of Lemma 3.11 of \cite{Sc07_hilbert} works exactly the same way in the presence of an $m$-adic filtration, as we have here, rather than a dyadic filtration as in \cite{Sc07_hilbert}.)
\end{proof}

\section{Flat balls}
This section is devoted to the proof of the following proposition.

\begin{prop}\label{flat}
For any $p>1$, we have 
$$\sum_{B\in \mathcal{G}_2} \beta_\Gamma(B)^p \diam(B) \leq C_p \length(\Gamma).$$
\end{prop}

Since $\mathcal{G}=\mathcal{G}_0 \cup \mathcal{G}_1 \cup \mathcal{G}_2$, combining Proposition \ref{flat}, Proposition \ref{nonflat}, Lemma \ref{G0lemma}, and Lemma \ref{reductionlemma} proves Theorem \ref{upperbound}.

To prove Proposition \ref{flat}, we follow roughly the outline in Section 4 of \cite{LS14}. This involves first proving some geometric results about the structure of arcs in $\mathcal{G}_2$-balls, and then using a geometric martingale argument, of the kind appearing in Section 4.2 of \cite{LS14}, as well as in \cite{Sc07_hilbert}, \cite{Sc07_metric}.

Given one of our filtrations $\mathcal{F}=\mathcal{F}^i$ defined above, let $\mathcal{G}_\mathcal{F}$ be the collection of all balls in $\mathcal{G}$ such that there is an arc $\tau'\in \mathcal{F}$ containing the center of $B$ such that $\tau'$ contains an element of $\Lambda(B)$ and $\diam(\tau')\leq C'\diam(B)$. 

\begin{lemma}
Every ball $B$ in $\mathcal{G}_0$ is in $\mathcal{G}_\mathcal{F}$ for at least one of our filtrations $\mathcal{F}\in\{\mathcal{F}^1, \dots, \mathcal{F}^{CJ}\}$.
\end{lemma}
\begin{proof}
Choose $k\in\mathbb{N}$ such that
$$ m^{-k-1} \leq \diam(2B) < m^{-k}. $$

Let $\{t_1<t_2<\dots < t_N\}$ be the points of $F_k$, as defined in Section \ref{filtrations}. Let $x=\gamma(t_x)$ be the center of $B$. Observe that we cannot have $t_x<t_2$ or $t_x>t_{N-1}$, since $10m B$ does not contain either endpoint of $\gamma$ as $B\in\mathcal{G}_0$.

It follows that $t_i \leq t_x \leq t_{i+1}$ for some $i\in\{2,\dots, N-1\}$. Since $d(\gamma(t_j), \gamma(t_{j+1}))\geq m^{-k}$ for $t_j\in F_k$, at least one of $\gamma(t_i)$ or $\gamma(t_{i+1})$ is not in $2B$. If $\gamma(t_i)\notin B$, then $\gamma([t_i, t_{i+2}])$ contains an arc of $\Lambda(B)$, and if $\gamma(t_{i+1})\notin B$ then $\gamma([t_{i-1}, t_{i+1}])$ contains an arc of $\Lambda(B)$.

Thus, there is an arc of $\Lambda(B)$ containing $x$ and contained in an arc of $\mathcal{F}^0_k$, the diameter of which is at most $2Cm^{-k}\leq 4Cm\diam(B)$. By Lemma \ref{filtration}, that arc of $\mathcal{F}^0_k$ is contained in a slightly larger arc belonging to one of our filtrations $\mathcal{F}=\mathcal{F}^i$. That filtration arc will have diameter at most $5Cm\diam(B)\leq 4C'\diam(B)$.

Therefore, $B\in\mathcal{G}_\mathcal{F}$.
\end{proof}

At this point, we fix a \textbf{single} filtration $\mathcal{F}=\mathcal{F}^i$ in the collection defined earlier, and a constant $C'>0$. To prove Proposition \ref{flat}, it suffices to prove, for any such filtration, that
$$\sum_{B\in \mathcal{G}_2 \cap \mathcal{G}_\mathcal{F}} \beta_\Gamma(B)^p \diam(B) \leq C_p \length(\Gamma).$$

\subsection{Cubes}\label{cubes}
In this section, we describe the construction of a system of cubes based on appropriate collections of balls in our space. This construction will be applied repeatedly with various parameters below.

We will follow the outline of Section 2.4 of \cite{LS14}, citing results from that paper when necessary. We observe that our results apply to $m$-adic scales whereas those results are stated in terms of dyadic scales, but this means only that the implied constants below depend on $m$.

Let $\mathcal{D}$ be a sub-collection of $2\mathcal{G} = \{2B : B\in\mathcal{G}\}$. Fix constants $R>0$ and $J'>0$ for the remainder of this subsection. 

Then by (an $m$-adic adjustment of) Lemma 2.14 of \cite{Sc07_metric} (which applies in any doubling metric space) there exists $P=P(X,R)\in\mathbb{N}$ such that $\mathcal{D}$ can be divided into collections $\mathcal{D}^1, \dots, \mathcal{D}^{PJ'}$ satisfying
$$ \dist(B, B') \geq R \rad(B) \text{ whenever } B, B'\in \mathcal{D}^i \text{ and } \rad(B)=\rad(B'),$$
and
$$ \rad(B)/\rad(B') \in m^{J'\mathbb{Z}} \text{ whenever } B, B'\in \mathcal{D}^i.$$
(See also Section 2.4 of \cite{LS14}.)

Now fix any $i\in \{1, \dots, PJ'\}$. Exactly as in Lemma 2.12 of \cite{LS14}, we can construct a system of ``m-adic cubes'' based on $\mathcal{D}^i$. Each ball $B$ in $\mathcal{D}^i$ yields a ``cube'' $Q(B)$ with the following properties:

\begin{lemma}\label{l:JAn-11-16}
Assume $J'$ is sufficiently large (depending on the data of $X$). Then
\begin{enumerate}[(1)]
\item If $B\in \mathcal{D}^i$, then $2B \subset Q(B) \subset 2(1+m^{-J'+2})B$.
\item If $B,B'\in\mathcal{D}^i$ and $Q(B)\cap Q(B')\neq\emptyset$ and $\rad(B)>\rad(B')$, then $Q(B')\subseteq Q(B)$.
\item If $B,B'\in\mathcal{D}^i$ have the same radius $r$, then $\dist(Q(B),Q(B'))>2(R-1)r$. 
\end{enumerate}
\end{lemma}

\subsection{Geometric lemmas about arcs}

\begin{lemma}\label{flatarc}
There is a constant $C_0$, depending only on the data of $X$, such that the following holds:
Let $B\in \mathcal{G}_2$ and let $\tau$ be any arc of $\gamma$ contained in $4C'B$. Then
\begin{equation}\label{flatarceqn}
\beta_\tau(4C'B) \leq C_0m\epsilon_0 \beta_\Gamma(B).
\end{equation}
\end{lemma}

We emphasize that the quantity on the lefthand side of \eqref{flatarceqn} is not $\tilde{\beta}(\tau)$ but rather $\beta_\tau$, i.e., the arc $\tau$ is being treated as a set in $X$. 

\begin{proof}[Proof of Lemma \ref{flatarc}]
Assume without loss of generality that $\beta_{\tau}(4C'B)>0$. For convenience, write $[a,b]\subset [0,1]$ for the domain of $\tau$.

Let $C_0 = 10(C+C_\eta)$.

Fix $k\in\mathbb{Z}$ such that
\begin{equation}\label{kdef}
m^{-k} < \beta_{\tau}(4C'B) \diam(4C'B)/C_0 \leq  m^{-k+1}.
\end{equation}

Consider $F_k = \{t_1 < t_2 < \dots < t_N\} \subset [0,1]$, as defined above, and write 
$$ F_k \cap [a,b] = t_j < t_{j+1} < \dots < t_\ell. $$
Note that (\ref{kdef}) implies that $\diam(\tau) \geq C_0m^{-k} \geq 10Cm^{-k}$ and therefore it follows from Lemma \ref{Fkdense} that $\ell\geq j+4$.

Consider $\gamma_0(t_j), \gamma_0(t_{j+1}), \dots, \gamma_0(t_\ell)$. We claim that for some $j\leq i \leq \ell-3$, there are three consecutive times $t_i < t_{i+1} < t_{i+2}$ in $F_k\cap [a,b]$ that get mapped out of (forward or backward) order by $\gamma_0$. In other words, there exist $t_i < t_{i+1} < t_{i+2}$ in $F_k \cap [a,b]$ such that neither
$$\gamma_0(t_i)<\gamma_0(t_{i+1})<\gamma_0(t_{i+2})$$
nor
$$\gamma_0(t_i)>\gamma_0(t_{i+1})>\gamma_0(t_{i+2})$$
holds.

Indeed, suppose there did not exist such $t_i, t_{i+1}, t_{i+2}$. Then we would have $\pi(\gamma(t_j)) < \pi(\gamma(t_{j+1})) < \dots < \pi(\gamma(t_\ell))$ (or vice versa). Let $\gamma_k = \pi^\infty_k \circ \gamma$. Then $\gamma_k(t_j), \dots, \gamma_k(t_\ell)$ are adjacent vertices of $X_k$ that form a monotone sequence.

Therefore, there is a monotone geodesic $L_k$ in $X_k$ passing through the points $\{\gamma_k(t_j), \gamma_k(t_{j+1}), \dots, \gamma_k(t_\ell)\}$. Let $L$ be any lift of $L_k$ to a monotone geodesic in $X$. By Lemma \ref{piproperties2}, we have that 
$$ \{ \gamma(t_j), \gamma(t_{j+1}), \dots, \gamma(t_\ell) \} \subset N_{C_\eta m^{-k}}(L). $$

By Lemma \ref{Fkdense}, a $2Cm^{-k}$-neighborhood of the set
$$ \{ \gamma(t_j), \gamma(t_{j+1}), \dots, \gamma(t_\ell) \} $$
contains $\tau$. Therefore a $(2C+C_\eta)m^{-k}$-neighborhood of $L$ contains $\tau$, contradicting the inequality $C_0 m^{-k} < \beta_\tau(4C'B) \diam(4C'B)$ from (\ref{kdef}).

So we have three consecutive $t_i < t_{i+1} < t_{i+2}$ in $F_k\cap[a,b]$ that get mapped out of order by $\gamma_0$. Therefore, we have
$$ \gamma_0(t_i) = \gamma_0(t_{i+2}) $$
and
$$ |\gamma_0(t_{i+1}) - \gamma_0(t_i) | = m^{-k}.$$

The arc $\gamma_0|[t_i,t_{i+2}]$ is in $\mathcal{F}^0_k$. Let $\tau'$ be a slightly larger arc of $\hat{\mathcal{F}}$ (not necessarily of $\mathcal{F}$) containing $\gamma_0|[t_i,t_{i+2}]$, as in Lemma \ref{filtration}. It is clear from Lemma \ref{filtration} that
$$ \tilde{\beta}(\tau') > 1/10.$$
In addition,
\begin{equation}\label{diamtaularge}
\diam(\tau') \geq m^{-k} \geq  \frac{1}{C_0 m} \beta_{\tau}(4C'B) \diam(4C'B) \geq  \frac{1}{C_0 m} \beta_{\tau}(4C'B) \diam(B)
\end{equation}

On the other hand, since $B\in \mathcal{G}_2$ and $\tau'\in\hat{\mathcal{F}}(B)$ with $\tilde{\beta}(\tau') > 1/10$, we must have
\begin{equation}\label{diamtausmall}
 \diam(\tau') \leq \epsilon_0 \beta_\Gamma(B) \diam(B).
\end{equation}

Combining \eqref{diamtaularge} and \eqref{diamtausmall}, it follows that
$$ \beta_\tau(4C'B) \leq C_0 m\epsilon_0 \beta_\Gamma(B).$$
\end{proof}

\begin{lemma}\label{nearestpoint}
Let $L$ be a monotone geodesic in $X$, and let $x\in X$. Let $x_L$ be the unique point on $L$ such that $\pi_0(x_L) = \pi_0(x)$. Then
$$\dist(x, L) \leq d(x,x_L) \leq 2\dist(x,L).$$
\end{lemma}
\begin{proof}
The first inequality is obvious. Let $p\in L$ be a point such that $\dist(x,L) = d(x,p)$. Then 
$$ |\pi_0(p) - \pi_0(x)| \leq d(x,p) = \dist(x,L). $$
Since $p$ and $x_L$ are on the same monotone geodesic, we have
$$ d(p,x_L) = |\pi_0(p) - \pi_0(x_L)| = |\pi_0(p) - \pi_0(x)| \leq d(p,x) = \dist(x,L).$$
So
$$ d(x,x_L) \leq d(x,p) + d(p, x_L) \leq 2\dist(x,L). $$
\end{proof}

\begin{lemma}\label{hausdorffclose}
Let $\tau$ be an arc in $X$ and let $L$ be a monotone geodesic such that $\tau \subset N_\epsilon(L)$. Let $s$ be the subsegment of $L$ such that $\pi_0(s) = \pi_0(\tau)$. Then 
\begin{equation}\label{hausdorffclose1}
\sup_{x\in \tau} \dist(x,s) < 2\epsilon
\end{equation}
and
\begin{equation}\label{hausdorffclose2}
\sup_{y\in s} \dist(y,\tau) < 2\epsilon
\end{equation}
\end{lemma}
\begin{proof}
If $x\in \tau$, let $x_L$ be the unique point of $L$ with $\pi_0(x_L)=\pi_0(x)$. Then $x_L\in s$ and $d(x,x_L) \leq 2\epsilon$ by Lemma \ref{nearestpoint}, which proves \eqref{hausdorffclose1}.

If $y\in s$, then there is a point $x\in \tau$ with $\pi_0(x)=\pi_0(y)$ and therefore again $d(x,y)\leq 2\epsilon$ by Lemma \ref{nearestpoint}. This proves \eqref{hausdorffclose2}.
\end{proof}

\begin{lemma}\label{pidiameter}
Let $B$ be a ball in $X$, let $h>0$, and let $\tau\subset B$ be an arc such that $\beta_\tau(B)\diam(B) < h$. Then 
$$ \diam(\pi_0 \circ \tau) \geq \diam(\tau) - 4h. $$
\end{lemma}
\begin{proof}
Let $L$ be a monotone geodesic such that $\tau \subset N_h(L)$. Let $s$ be the sub-segment of $L$ such that $\pi_0(s) = \pi_0(\tau)$. Note that $s$ is connected, because $\tau$, and therefore $\pi_0(\tau)$, is connected.

By Lemma \ref{hausdorffclose}, the Hausdorff distance between $\tau$ and $s$ is at most $2h$. So
$$ \diam(\tau) \geq \diam(s) - 4h = \diam(\pi_0(s)) - 4h = \diam(\pi\circ \tau) - 4h. $$
\end{proof}

\begin{lemma}\label{overlap}
Let $\gamma([a,b])$ be contained a ball $B\subset X$. Suppose that for some $c\in[a,b]$ and $\delta\in(0,1)$, 
$$ | \gamma_0([a,c]) \cap \gamma_0([c,b]) | > \delta. $$
Then there is an arc $\tau\in \hat{\mathcal{F}}(B)$ such that $\tilde{\beta}_{\gamma_0}(\tau) > 1/10$ and $\diam(\tau) > \delta/5m$.
\end{lemma}
\begin{proof}
Let $k\in\mathbb{N}$ be such that
$$ \delta/5m <  m^{-k} \leq \delta/5.$$

To prove the lemma, it suffices to show that there exist consecutive points $t_i < t_{i+1} < t_{i+2}$ in $F_k \cap [a,b]$ such that $\gamma_0(t_i) = \gamma_0(t_{i+2})$.

Suppose not. Let $t_j, t_{j+1}, \dots, t_\ell$ enumerate the points of $F_k \cap [a,b]$ in order. (Note that there must be at least two of them, by our assumptions and Lemma \ref{Fkdense}.) Then we have
$$ \gamma_0(t_j) < \gamma_0(t_{j+1}) < \gamma_0(t_{j+2}) < \dots < \gamma_0(t_{\ell}), $$
or the reverse order.

Choose $j\leq i \leq k$ such that 
$$ t_i \leq c < t_{i+1}. $$

Then, using Lemma \ref{Fkdense}, we see that
$$ \gamma_0([a,c]) \subset N_{2m^{-k}}(\gamma_0(\{t_1, \dots, t_i\})) \subset N_{2m^{-k}}([\gamma_0(t_1), \gamma_0(t_i)])$$
and
$$ \gamma_0([c,b]) \subset N_{2m^{-k}}(f(\{t_{i+1}, \dots, t_\ell\})) \subset N_{2m^{-k}}([\gamma_0(t_{i+1}), \gamma_0(t_{\ell})])$$

It follows that
$$ |\gamma_0([a,c]) \cap \gamma_0([c,b])| \leq 4m^{-k} < \delta, $$
which is a contradiction.
\end{proof}

\begin{lemma}\label{bigdiameter}
Let $B=B(x,r)\in \mathcal{G}_2$ and let $\tau\in\Lambda(B)$ be an arc containing $x$. Let $h=4C_0 C'm\epsilon_0\beta_\Gamma(B)\diam(B)$. 

Then there is a monotone segment of diameter at least $4r-20h$ within Hausdorff distance $2h$ of $\tau$.

In particular,
$$ \diam(\tau) \geq 4r - 30h. $$
\end{lemma}
\begin{proof}
By Lemma \ref{flatarc} we have
$$ \beta_\tau(4C'B)\diam(4C'B) < h.$$
So by Lemma \ref{hausdorffclose} there is a monotone segment $s$ such that $\pi_0(s) = \pi_0(\tau)$ and $s$ is within Hausdorff distance $2h$ of $\tau$.

We want to show that $\diam(s) \geq 4r-20h$. Suppose not. Then
\begin{equation}\label{pitaudiam}
\diam(\pi_0(\tau)) = \diam(\pi_0(s)) = \diam(s) < 4r-20h.
\end{equation}

Suppose $\tau$ is parametrized by $\gamma$ on $[a,b]$, so that $\gamma(a)$ and $\gamma(b)$ are on $\partial (2B)$ and $\gamma(c) = x$.

Observe that $\gamma ([a,c])$ and $\gamma ([c,b])$ each have diameter at least $2r$. Therefore, by Lemma \ref{pidiameter}, $\pi_0 \circ \gamma ([a,c])$ and $\pi_0 \circ \gamma ([c,b])$ each have diameter at least $R= 2r-4h$.

Therefore, (\ref{pitaudiam}) implies that $|\gamma_0([a,c]) \cap \gamma([c,b])|\geq 12h$.

It follows from Lemma \ref{overlap} that there is $\tau\in\hat{\mathcal{F}}(B)$ of diameter at least
$$12h/5m = \frac{48C_0 C'}{5}\epsilon_0\beta_\Gamma(B)\diam(B) > \epsilon_0\beta_\Gamma(B)\diam(B)$$ 
with $\tilde{\beta}(\tau)>1/10$. 

This contradicts our assumption that $B\in \mathcal{G}_2$.
\end{proof}

\begin{lemma}\label{farflatarc}
Assume $\epsilon_0<1/(400C_0 C'm)$. Let $B\in\mathcal{G}_2$ have radius $r$. Let $\tau\in \Lambda(B)$ contain the center of $B$, and let $\tau'$ be an arc of $\mathcal{F}$ containing $\tau$ and contained in $4C'B$. Then there is a sub-arc $\xi$ in $\Gamma\cap 2B$ such that
$$ \diam(\xi) > 200C_0C'm\epsilon_0 \beta_\Gamma(B)\diam(B) $$
and
$$ \dist(\xi, \tau') > 10\epsilon_0 \beta_\Gamma(B)\diam(B).$$
\end{lemma}
\begin{proof}
Assume without loss of generality that $\beta_\Gamma(B)>0$. Let $L$ be a monotone geodesic such that 
$$ \sup_{y\in \tau'}\dist(y,L) < \beta_{\tau'}(4C'B)\diam(4C'B) < 4C_0C'm\epsilon_0 \beta_\Gamma(B) \diam(B), $$
where the last inequality is from Lemma \ref{flatarc}.

There is a point $x\in \Gamma \cap B$ such that
$$ \dist(x,L) > 400C_0C'm\epsilon_0 \beta_\Gamma(B)\diam(B).$$
Indeed, if not, then
$$ \Gamma \cap B \subset N_{400C_0C'm\epsilon_0 \beta_\Gamma(B)\diam(B)}(L),$$
in which case
$$ \beta_\Gamma(B) \diam(B) \leq 400C_0C'm\epsilon_0 \beta_\Gamma(B)\diam(B)$$
which is a contradiction as $\epsilon_0<1/(400C_0C'm)$.

It immediately follows that $\dist(x,\tau')\geq 300C_0C'm\epsilon_0 \beta_\Gamma(B)\diam(B)$. 

Now let $\xi$ be a sub-arc of $\Gamma$ that contains $x$, stays in $B(x, 250C_0C'm\epsilon_0 \beta_\Gamma(B)\diam(B))$, and touches the boundary of $B(x, 250C_0C'm\epsilon_0 \beta_\Gamma(B)\diam(B))$.
\end{proof}

\begin{prop}\label{covering}
Let $B\in\mathcal{G}_\mathcal{F}$ be a ball of radius $r$. Let $\tau\in \Lambda(B)$ contain the center of $B$, and let $\tau'\in\mathcal{F}$ contain $\tau$ and be contained in $4C'B$. Let $\xi$ be a sub-arc of $\Gamma \cap 2B$ as in Lemma \ref{farflatarc}. Let $E = (\tau' \cup \xi) \cap 2B$.

Suppose we cover $E$ by balls $\{B_i\}$ such that $\diam(B_i) < 10\epsilon_0 \beta_\Gamma(B)\diam(B)$. Then
$$\sum_i \diam(B_i) \geq 4r + \epsilon_0 \beta_\Gamma(B)\diam(B).$$
\end{prop}
\begin{proof}
Note that each ball $B_i$ can intersect at most one of $\tau'$ or $\xi$. So
$$ \sum_i \diam(B_i) \geq \diam(\tau') + \diam(\xi).$$

We know $\diam(\xi) > 200C_0C'm\epsilon_0 \beta_\Gamma(B)\diam(B)$. We also know from Lemma \ref{bigdiameter} that
$$ \diam(\tau')\geq \diam(\tau) > 4r - 120C_0C'm\epsilon_0 \beta_\Gamma(B)\diam(B).$$

Putting these together proves the proposition.
\end{proof}

\subsection{A geometric martingale}\label{ss:mart}
Now fix an integer $M\geq 0$ and consider the collection
$$ \mathcal{G}_{\mathcal{F},M} = \{B\in \mathcal{G}_\mathcal{F} : \beta_\Gamma(B)\in [m^{-M},m^{-M+1})\}$$

We apply the construction of subsection \ref{cubes} to $\mathcal{G}_{\mathcal{F},M}$, with parameters $R=3$ and $J'$ being the smallest integer larger than $M -\log_m(10\epsilon_0)+20$. This firsts separates $\mathcal{G}_{\mathcal{F},M}$ into $PJ'$ different collections 
$$ \mathcal{G}^1_{\mathcal{F},M}, \dots, \mathcal{G}^{PJ'}_{\mathcal{F},M} $$
(where $P=P(X)$), with the properties outlined in subsection \ref{cubes}

For each $t\in\{1, \dots, PJ'\}$, that construction then assigns a collection $\Delta_t$ of cubes associated to the balls of $\mathcal{G}^t_{\mathcal{F},M}$ and having the properties outlined in Lemma \ref{l:JAn-11-16}. (The collection $\Delta_t$ also depends on the filtration $\mathcal{F}$ and the integer $M$, but we suppress these in the notation for convenience.)

For $Q\in \Delta_t$ we write
\begin{equation}\label{e:Q-decomposition}
 Q=R_Q\cup \left(\cup_j Q^j\right),
\end{equation}
where
$Q^j\in \Delta_t$  is in maximal such that $Q^j\subset Q$, and  $R_Q=\Gamma\setminus\left(\cup_j Q_j\right)$.
We will use this decomposition to construct a geometric martingale.

Combining Proposition \ref{covering} with Lemma \ref{l:JAn-11-16} gives the following corollary, exactly as in Proposition 4.7 of \cite{LS14}.

\begin{cor}\label{c:four-seven}
Suppose $Q=Q(B)\in \Delta_t$  and $Q=R_Q\cup \left(\cup_j Q^j\right)$ as in \eqref{e:Q-decomposition}.
Then 
 $$\mathcal{H}^1(R_Q)+ \sum_j \diam(Q^j)  \geq \diam(Q)(1+ \frac1{10}\epsilon_0 m^{-M}.)$$
\end{cor}

We now repeat Proposition 4.8 of  \cite{LS14} verbatim (except for adjusting dyadic scales to $m$-adic scales).
The proof is the same, and hinges on our Corollary  \ref{c:four-seven} in place of  Proposition 4.7 of \cite{LS14}.
Let $\Delta=\Delta_t$ for $t\in \{1,...,PJ'\}$.
\begin{prop}\label{p:martingale-prop}
$$\sum\limits_{Q\in\Delta_t} \diam(Q) \leq \frac{20}{\epsilon_0} m^M \cH^1(\Gamma)$$
\end{prop}
We observe that this gives (by considering all possible values of $t$ and inserting the value for $\beta_\Gamma$)
\begin{align*}
\sum_{B\in \mathcal{G}_2 \cap \mathcal{G}_\mathcal{F}} \beta_\Gamma(B)^p \diam(B) &\leq \sum_{M=1}^\infty \sum_{t=1}^{PJ'} \sum_{B\in \mathcal{G}^t_{\mathcal{F},M}} \beta_\Gamma(B)^p \diam(B)\\
&\leq \sum_{M=1}^\infty 2P(M-\log(10/\epsilon_0)+20) m^{-(M+1)p} \sum_{B\in \mathcal{G}^t_{\mathcal{F},M}} \diam(Q(B))\\
&\leq \sum_{M=1}^\infty 2P(M-\log(10/\epsilon_0)+20) m^{-(M+1)p} \frac{20}{\epsilon_0} m^M \cH^1(\Gamma)\\
&\leq C_{p, X} \cH^1(\Gamma)
\end{align*}
for $p>1$.
Thus we will have   Proposition \ref{flat}  as soon as we complete the 
proof of Proposition \ref{p:martingale-prop}.
\begin{proof}[Proof of Proposition \ref{p:martingale-prop}]
In the same manner as \cite{Sc07_hilbert, Sc07_metric, LS14} 
we define positive function $w_Q:X\to \mathbb{R}$ such that
\begin{enumerate}[(i)]
\item
$\int_Q w_Qd\cH^1_\Gamma \geq \diam(Q)$
\item
For almost all $x\in \Gamma$, 
$$\sum\limits_{Q\in \Delta} w_Q(x)\leq \frac{20}{\epsilon_0} m^M$$
\item $w_Q$ is supported inside $Q$
\end{enumerate}
The functions $w_Q$ will be constructed as a martingale.
Denote $w_Q(Z)=\int_Z w_Qd\cH^1_\Gamma$.
Set
$$
w_Q(Q)=\diam(Q) .
$$
Assume now that $w_Q(Q')$ is defined.  We define $w_Q(Q'^i)$ and $w_Q(R_{Q'})$,
where
$$ 
Q'=(\cup  Q'^i) \cup R_{Q'},
$$
a decomposition as given by equation \eqref{e:Q-decomposition}.

Take
$$
w_Q(R_{Q'})=\frac{w_Q( Q')}{s'} \cH^1_\Gamma(R_{Q'}) 
$$
(uniformly distributed)
and
$$
w_Q( Q'^i)=\frac{w_Q(Q')}{s'}\diam(Q'^i),
$$
where 
$$
s'=\cH^1_\Gamma(R_{Q'})+\sum_i \diam(Q'^i).
$$
This will give us $w_Q$.
Note that 
$s'\leq 2 \cH^1(\Gamma\cap Q')$.
Clearly (i) and (iii) are satisfied.
Furthermore, 
If $x\in R_{Q'}$, we have from (a rather weak use of) Corollary   \ref{c:four-seven} that
\begin{equation}\label{e:r-Q-estimate}
w_Q(x)\leq \frac{w_Q(Q')}{s'}\leq \frac{w_Q(Q')}{\diam(Q')}\,. 
\end{equation}

To see (ii), note that for any $j$ we may write:

\begin{eqnarray*}
\frac{w_Q( Q'^{j})}{\diam(Q'^{j}) }
&=&
\frac{w_Q( Q')}{s'}\\
&=&
\frac{w_Q( Q')}{\diam(Q') }
\frac{\diam(Q' )}{s'}\\
&=&
\frac{w_Q( Q')}{\diam(Q') }
\frac{\diam(Q') }
	{\cH^1_\Gamma(R_{Q'}) + 
		\sum\limits_{i} \diam(Q'^i) }\\
&\leq&
\frac{w_Q( Q')}{\diam(Q') }
\frac{1}
	{1+ c_0m^{-M}}\,,\\
\end{eqnarray*}
where $c_0=\frac1{10}\epsilon_0$ is obtained from  Corollary  \ref{c:four-seven}.

And so,
\begin{eqnarray*}
\frac{w_Q( Q'^{j})}{\diam(Q'^{j})} \leq 
 	q   \frac{w_Q( Q')}{\diam(Q')}
\end{eqnarray*}
with $q=\frac{1}
{1+ c_0m^{-M}}$.  Now, suppose 
that  $x\in Q_N \subset ...\subset Q_1$.
we  get:
\begin{eqnarray*}
\frac{w_{Q_1}(Q_N)}{\diam(Q_N)} &\leq& 
  q\frac{w_{Q_1}(Q_{N-1})}{\diam(Q_{N-1})} \\
  &\leq&...\\
  &\leq&
  q^{N-1}\frac{w_{Q_1}(Q_{1})}{\diam(Q_1)}=q^{N-1}.
\end{eqnarray*}
We have using \eqref{e:r-Q-estimate} that for 
$x\in R_{Q_{N}}$
\begin{equation}
w_{Q_1}(x) \leq   \frac{w_{Q_1}(Q_N)}{\diam(Q_N)} \leq  q^{N-1}. 
\end{equation}
Let $E$ denote the collection of all  elements $x$ which are in an infinite sequence of $\Delta$ i.e. can be written as elements $x\in .... \subset Q_N \subset ...\subset Q_1$,  for any positive integer $N$.
Then,
as $\cH^1_\Gamma(Q)\geq r(B(Q))\geq \frac15\diam(Q)$,
we have that for any $N$
\begin{equation}
w_{Q_1}(Q_N) \leq   \diam(Q_N) q^{N-1} \leq  5q^N\cH^1_\Gamma(Q_N)
\end{equation}
which yields that for $\cH^1_\Gamma$-almost-every $x\in E$ we have that $w_{Q_1}(x)=0$.

This will give us (ii) as a sum of a geometric series  since
$$\sum q^n = \frac1{1-q}
\leq  1+ \frac1{c_0 m^{-M}}=\frac{20}{\epsilon_0}m^{M}.$$

Now,
\begin{eqnarray*}
\sum\limits_{Q \in \Delta}\diam(Q)
&=&
\sum\limits_{Q \in \Delta}\int w_{Q}(x)d\cH^1_\Gamma(x)\\
&=&
\int \sum\limits_{Q \in \Delta} w_{Q}(x)d\cH^1_\Gamma(x)\\
&\leq&
\frac{20}{\epsilon_0}
\int  m^{M} d\cH^1_\Gamma(x)
=
\frac{20}{\epsilon_0}
m^{M}\cH^1(\Gamma).
\end{eqnarray*}

\end{proof}

\section{Counterexample for $p=1$}\label{counterexample} 
In this section, we show that Theorem \ref{upperbound} is false if $p=1$. The constructed counterexample will essentially be the same as Example 3.3.1 in \cite{Sc07_survey}, suitably imported to one of the metric spaces of this paper. 

Consider the admissible inverse system of Example \ref{langplaut}, and let $X$ denote its limit. In this case, $m=4$, and so the $n$th stage $X_n$ of the construction is a graph with edges of length $m^{-n}$.

Let $V_k(X_n)$ or $V_k(X)$ denote the $k$th level vertices of $X_n$ or $X$, respectively, as defined in Definition \ref{levelvertices}. In addition, for $k\geq 1$, define
$$\tilde{V}_k(X_n) = V_k(X_n) \setminus V_{k-1}(X_n).$$

Fix an arbitrary constant $A>1$. For each $j\in \mathbb{N}$, let $\tilde{\mathcal{B}}_j$ denote the collection of balls $\{B(v,A4^{-k})\}$ in $X$, where $v$ runs over $V_k(X)$. Let $\tilde{\mathcal{B}} = \cup_j \tilde{\mathcal{B}}_j$.

If $\mathcal{B}$ is as in Theorem \ref{upperbound}, then it is easy to see that, for all compact, connected sets $\Gamma\subset X$, the sums
$$ \sum_{B\in \tilde{\mathcal{B}}} \beta_\Gamma(B) \diam B \text{ and } \sum_{B\in \mathcal{B}}  \beta_\Gamma(B) \diam B $$
are comparable, up to an absolute constant depending only on $A$ and the data of $X$.

Therefore, to show that Theorem \ref{upperbound} is false for $p=1$, it suffices to show the following.

\begin{prop}
There is no constant $C>0$ such that
$$ \sum_{B\in \tilde{\mathcal{B}}}  \beta_\Gamma(B) \diam(B) \leq C\mathcal{H}^1(\Gamma)$$
for all rectifiable curves $\Gamma\subset X$. 
\end{prop}

\begin{proof}
For each integer $N>4$, we will construct a curve $\Gamma_N$ in $X$ such that $\mathcal{H}^1(\Gamma_N) = 2$ and 
\begin{equation}\label{betalog}
\sum_{B\in \tilde{\mathcal{B}}}  \beta_{\Gamma_N}(B) \diam(B) \geq c\log(N).
\end{equation}
for some constant $c$ depending only on $A$. This will prove the proposition.

Given $n\geq 0$ and $\epsilon \in (0,1/4)$, we define $\Gamma(n,\epsilon) \subset X_n$ as follows: $\Gamma(n, \epsilon)$ consists of the union of a monotone geodesic $L$ with one ``spike'' of length $\epsilon$ emerging from a vertex of $L\cap\pi_0^{-1}(1/4)$, four ``spikes'' of length $\epsilon/4$ emerging from $L \cap \pi_0^{-1}(\{1/16, 5/16, 9/16, 13/16\})$, up until $4^n$ ``spikes'' of length $\epsilon/4^n$, as illustrated in the figure \ref{counterexpic}.

\begin{figure}
\includegraphics[scale=0.4]{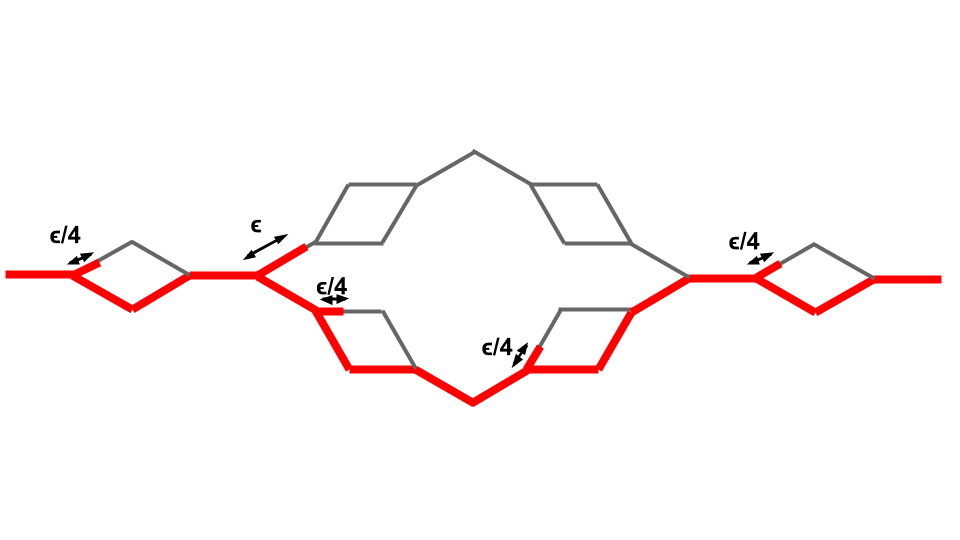}
\caption{The red curve is $\Gamma(2,\epsilon)$.}
    \label{counterexpic}
\end{figure}

It follows that
$$ \mathcal{H}^1(\Gamma(n,\epsilon)) = 1 + n\epsilon. $$

For each integer $N>4$, we can define a rectifiable curve $\Gamma_N \subset X$ as any one of the obvious lifts to $X$ of the curve $\Gamma(N, 1/N)$ in $X_N$.

We then see that 
$$ \mathcal{H}^1(\Gamma_N) = 2. $$

In addition, it is clear in this example that
$$ \sum_{B\in \tilde{\mathcal{B}}}  \beta_{\Gamma_N}(B) \diam(B) = \sum_{j=1}^N\sum_{v\in \tilde{V}_j(X_N)}\sum_{i=j}^\infty\beta_{\Gamma(N,1/N)}(B(v,A4^{-i})) \diam\left(B(v,A4^{-i})\right),$$
where the sum on the right is taken over vertices in $X_N$ and the $\beta$-numbers on the right are taken with respect to monotone geodesics in $X_N$. In other words, to show \eqref{betalog} about the curve $\Gamma_N$ in $X$, it suffices to look at $\beta$-numbers in $X_N$ with respect to $\Gamma(N,1/N)\subset X_N$.

Let $\epsilon=1/N$. Fix a vertex $v\in L\cap \tilde{V}_j(X_N)$ from which a spike of length $\epsilon/4^j$ emerges. We see that
$$ \sum_{i=j}^\infty  \beta_{\Gamma(N,1/N)}(B(v,A4^{-i})) \diam \left(B(v,A4^{-i})\right) \geq c\epsilon 4^{-j}\log(1/\epsilon), $$
where $c$ is a constant depending on $A$.

Therefore, 
\begin{align*}
\sum_{B\in \tilde{\mathcal{B}}}  \beta_\Gamma(B) \diam(B) &\geq c\left(\epsilon \log(1/\epsilon) + 4\frac{\epsilon}{4}\log(1/\epsilon) + \dots + 4^N\frac{\epsilon}{4^N}\log(1/\epsilon)\right)\\
&\geq c\epsilon N \log(1/\epsilon)\\
&= c\log(N),
\end{align*}
because $\epsilon=1/N$.

This completes the proof.
\end{proof}

\part{The construction}\label{constructionpart}

In this part, we prove Theorem \ref{construction}. Let $X$ be the limit of an admissible inverse system, as in Definition \ref{admissiblesystem}.

The structure of this part of the paper is as follows: After some additional setup in Section \ref{additionalsetup}, we state the key intermediate step, Proposition \ref{constructionprop}, in Section \ref{constructionpropsection} and give its proof in Sections \ref{liftingalgorithm} through \ref{propositionproof}. Finally, the proof of Theorem \ref{construction} is given in Section \ref{theoremproof}.

\section{Some additional setup}\label{additionalsetup}

It will be convenient in Part \ref{constructionpart} to assume that the scale factor $m$ is sufficiently large, depending on the other parameters $\eta, \Delta$ of the space $X$ (see Definition \ref{admissiblesystem}). To achieve this, we will ``skip'' some scales, using the following lemma.

\begin{lemma}\label{mlargelemma}
Fix any integer $K\geq 1$. For each $n\geq 0$, let $\tilde{X}_n = X_{Kn}$ and let $\tilde{\pi}_n:\tilde{X}_{n+1}\rightarrow \tilde{X}_n$ be defined by composing $\pi_n \circ \pi_{n+1} \circ \dots \circ \pi_{nK-1}$.

Then $\tilde{X}_n$ forms an admissible inverse system with $m$ replaced by $m^K$, the same constant $\Delta$, and $\eta$ replaced by $4\eta$. The inverse limit of $\tilde{X}_n$ is isometric to the inverse limit of $X_n$, and the spaces $\tilde{X}$, $\{\tilde{X}_n\}$ are uniformly doubling with the same doubling constant as $X$, $\{X_n\}$.
\end{lemma}
\begin{proof}
Since $\{\tilde{X}_n\}$ is a subsequence of $\{X_n\}$, it has the same limit. The fact that $m$ is replaced by $m^K$ and $\Delta$ remains unchanged is clear. Furthermore, since we consider the same spaces, the fact that the doubling constant is unchanged is also clear.

To see what happens to $\eta$, consider an arbitrary point $x_n\in \tilde{X}'_n \subset X_{Kn}$. Using Lemma \ref{piproperties}, we see that
$$ \diam((\tilde{\pi}_n)^{-1}(x_n)) \leq 2\eta(m^{-Kn} + m^{-(Kn+1)} + \dots + m^{-(K(n+1))-1}) \leq 2\eta m^{-Kn} \left(\frac{m}{m-1}\right) \leq 4\eta (m^K)^{-n}.$$
\end{proof}

Therefore, in proving Theorem \ref{construction}, we may without loss of generality assume that $m$ is large compared to $\eta$ and $\Delta$. In fact the necessary assumption will be that 
\begin{equation}\label{mlarge}
C_{\eta,\Delta}m^{-1}\frac{\log(20C_\eta m)}{\log(m)} < \frac{1}{100},
\end{equation}
where $C_{\eta,\Delta}$ is a constant depending only on $\eta$ and $\Delta$ that will arise in the proof. This can be achieved by making $m$ sufficiently large, i.e., by replacing $m$ by $m^K$, where $K$ is the smallest integer such that \eqref{mlarge} holds with $m^K$.  This assumption is in force from now on.

Recall from subsection \ref{mainresults} our choice of $A>1$, which depends only on $\eta$. We fix a small constant $\epsilon$ which depends on $m,A,\eta,\Delta$. Taking $\epsilon = (10C_\eta A m \delta)^{-10}$ suffices.

From now on, constants which depend only on $\eta$ and $\Delta$ (but not $m$) will be denoted by $C_\eta$ (if only depending on $\eta$) or $C_{\eta,\Delta}$. Constants which may depend on on $m$ as well are denoted $C_X$.

Given an edge $e_n$ in $X_n$ we set
\begin{align*}
Q(e_n) &= \{x\in X : \dist(\pi_n(x), e_n)\leq 2m^{-n}\}\\
&=\{x\in X: \pi_n(x) \text{ lies within two edges of } e_n\}.
\end{align*}
Observe that if $e'_n$ is an edge of $X_n$ that shares both endpoints with $e_n$, then $Q(e_n)=Q(e'_n)$.

We set $\mathcal{Q}_n$ to be the collection of all $Q(e_n)$ as $e_n$ ranges over the edges of $X_n$, and $\mathcal{Q}=\bigcup_{n=0}^\infty \mathcal{Q}_n$. Recall our choice of $A\geq 100C_\eta$ from subsection \ref{mainresults}. Thus, observe that if $e_n$ is an edge of $X_n$, then $Q(e_n)\subset B(e_n)$ and
$$\diam (Q(e_n)) \geq m^{-n} \geq \frac{1}{2A}\diam(B(e_n)).$$
For $Q(e_n)\in \mathcal{Q}$, we set
$$ \beta_E(Q(e_n)) = \frac{1}{\diam (B(e_n))}\inf_{L} \sup_{x\in E\cap Q(e_n)} \dist(x,L) \leq \beta_E(B(e_n)).$$
It will therefore suffice to control the length of the curve $G$ in Theorem \ref{construction} by
$$ \mathcal{H}^1(G) \leq C\left(\diam(E) + \sum_{Q\subset \mathcal{Q}, \beta_E(Q)\geq\epsilon} \diam(Q)\right), $$
which is what we will actually do below.

For each $Q(e_n)\in\mathcal{Q}$, there is at least one monotone geodesic in $X$ which achieves the minimum in $\beta_E(Q(e_n))$. Although it is not unique, we will fix one such monotone geodesic for each $Q(e_n)$ and call it ``the'' optimal monotone geodesic for $Q(e_n)$.

As a final notational convenience, we will use the more compact notation $|E|$ below to denote $\mathcal{H}^1(E)$ for a subset $E$ of $X$ or any $X_n$.

\section{The main preliminary step: Proposition \ref{constructionprop}}\label{constructionpropsection}
The main part of the proof of Theorem \ref{construction} is encapsulated in the following proposition, which is then iterated to produce the curve $G$ of Theorem \ref{construction}. Although technical to state, the idea of Proposition \ref{constructionprop} is that it splits the set $E$ into a portion $E\cap\Gamma$ contained in our ``first pass'' at a curve $\Gamma$, and subsets $E(e_n)\subset E\setminus \Gamma$ (indexed by edges of $X_n$ for varying $n\in\mathbb{N}$) on which we will repeat the construction.

\begin{prop}\label{constructionprop} 
Let $E\subset X$ be a compact set such that $\pi_{n_0}(E)$ is contained in an edge $e_{n_0}\in X_{n_0}$, for some $n_0\geq 0$. Assume that, for every $n\in\mathbb{N}$, $\pi_n(E)$ does not contain any vertex or edge-midpoint of $X_n$.

Then there is a compact connected set $\Gamma\subset X$, a sub-collection $\mathcal{Q}_\Gamma\subseteq \mathcal{Q}$, and sets $E(e_n)$ ($n_0+2\leq n\in\mathbb{N}$, $e_n$ an edge of $X_n$) whose union contains $E\setminus \Gamma$, such that:
\begin{enumerate}[(i)]
\item \label{gammalength} $|\Gamma| \leq 2m^{-n_0} + C_{X}\sum_{Q\subset \mathcal{Q}_\Gamma, \beta_E(Q)\geq\epsilon} \diam(Q)$
\item \label{gammanearE} $\pi_{n_0}(\Gamma)\supseteq e_{n_0}$.
\item \label{gammaprojection} If $Q(e_n)\in \mathcal{Q}_\Gamma$ for some edge $e_n$ of $X_n$, then $\pi_{n_0}(e_n)\subset e_{n_0}$.
\item \label{bubblenotingamma} For each $E(e_n)\neq\emptyset$, we have $\pi_n(E(e_n))\subset e_n$, $\pi_{n_0}(e_n)\subset e_{n_0}$, and $\pi_n(e_\ell)\not\subseteq e_n$ for any $e_\ell$ with $Q(e_\ell)\in \mathcal{Q}_\Gamma$. In particular, $\diam(E(e_n))\leq C_\eta m^{-n}$.
\item \label{bubblesdisjoint} If $E(e_n)$ and $E(e'_{n'})$ are non-empty and $n\leq n'$, then $\pi_n(e'_{n'})\not\subset e_n$. 
\item \label{bubbleneargamma} $\dist(E(e_n)), \Gamma) \leq C_{\eta}m^{-(n-1)}$, 

\item \label{bubblesum} $\sum_{n\geq n_0} \sum_{e_n: E(e_n)\neq\emptyset} |e_n| \leq \frac{C_{\eta,\Delta}}{m^2}|\Gamma| + C_{X}\sum_{Q\subset \mathcal{Q}_\Gamma, \beta_E(Q)\geq\epsilon} \diam(Q)$. 

\end{enumerate}
\end{prop}

We now begin to prove Proposition \ref{constructionprop}. This is done by constructing preliminary (disconnected) sets $\Gamma_n\subset X_n$ in Section \ref{liftingalgorithm}, augmenting these into connected sets in Section \ref{connectability}, and showing that these connected sets converge to a set $\Gamma\subset X$ with the above properties, in Sections \ref{convergence} and \ref{propositionproof}.

\section{The lifting algorithm}\label{liftingalgorithm}
In this section, we will construct, for each $n\geq n_0$, a (not necesarily connected) set $\Gamma_n$ in $X_n$. The sets $\Gamma_n$ will be simplicial, i.e., for each $n\geq n_0$, $\Gamma_n$ will be a union of edges of $X_n$. The sets $\Gamma_n$ will then be augmented to become connected sets in Section \ref{connectability}, and the limit of those augmented sets will be the continuum $\Gamma\subset X$ of Proposition \ref{constructionprop}.

The sets $\Gamma_n$ will be constructed inductively, so we will construct $\Gamma_{n_0}$ and then describe how to construct $\Gamma_{n+1}\subset X_{n+1}$ given $\Gamma_n\subset X_n$. 

We first make some more definitions.

For the purposes of this section, a \textit{monotone set} in $X$ or some $X_n$ is a set $E$ such that $E \cap \pi_0^{-1}(t)$ contains at most one point for all but a finite number of points $t\in [0,1]$, and at most two points for all $t\in [0,1]$. Of course, a monotone geodesic segment is a monotone set, but a monotone set need not be connected. As an example, one should think of a union of finitely many monotone geodesic segments $L_1, L_2, \dots, L_k$ that project onto intervals in $X_0\cong [0,1]$ that have disjoint interiors.

\begin{definition}\label{Sdefinition}
For an edge $e_n$ in $X_n$, we define a subset $S(e_n)\subset X$ as follows: Consider each edge $e_{n+1}$ of $X_{n+1}$ that projects into $e_n$. Take a single connected lift $\tilde{e}_{n+1}$ to $X$ of each such $e_{n+1}$, with $|\tilde{e}_{n+1}|=m^{-(n+1)}$. We call the union of all these lifts $S(e_n)$.
\end{definition}
Observe that $S(e_n)$ is a compact (though not necessarily connected) subset of $X$, and satisfies $|S(e_n)|\leq C_X m^{-n}$ (see Lemma \ref{piproperties}). Moreover, by using Lemma \ref{piproperties}, we observe that $S(e_n)$ is even contained in connected subset of $X$ with $\mathcal{H}^1$-measure at most $C'_X m^{-n}$.

\begin{definition}\label{goodedgedef}
For each $n\geq n_0$, we make the following definitions:
\begin{enumerate}[(1)]
\item An edge $e_n$ in $X_n$ is called \textit{good} if $\beta_E(Q(e_n))<\epsilon$ and
$$ B_{X_n}(v_n, (1-2A\epsilon)m^{-n}) \cap \pi_n(E_n) \cap \pi_0^{-1}(\pi_0(e_n)) \neq \emptyset$$
and
$$ B_{X_n}(w_n, (1-2A\epsilon)m^{-n}) \cap \pi_n(E_n) \cap \pi_0^{-1}(\pi_0(e_n)) \neq \emptyset,$$
where $v_n$ and $w_n$ are the endpoints of $e_n$.
\item An edge $e_n$ in $X_n$ is called \textit{bad} if it is not good.
\item A vertex $v'_n$ of level $i\geq n+1$ in $X_n$ is called \textit{special} if there exists an edge $e_n\subset X_n$ containing $v'_n$ and a point $x\in E$ such that
$$ \pi_0(x) \subset [\pi_0(v'_n) - \frac{1}{2}m^{-i}, \pi_0(v'_n) + \frac{1}{2}m^{-i}] $$  
and
$$ d(\pi_n(x), e_n \cap \pi_0^{-1}(\pi_0(x))) < m^{-n}.$$
\end{enumerate}
\end{definition}

Let $\Gamma_{n_0} = e_{n_0}$, the edge containing the full projection of $E$ (by assumption). If $\beta_E(Q(e_{n_0}))<\epsilon$ it is convenient to also call $e_{n_0}$ a good edge, even if it does not satisfy the other criteria of the definition above. Observe that there is a choice of optimal geodesic $L$ for $Q(e_{n_0})$ such that $\pi_{n_0}(L)\supseteq e_{n_0}$.

\begin{rmk}\label{goodedgeremark}
The following simple observation is the most important property of a  good edge: Suppose $e_n$ is a good edge in $X_n$ and $Q\in\mathcal{Q}_n$ has $\beta_E(Q)<\epsilon$ and $e_n\subset\pi_n(Q)$. Then $\pi_n(L)$ contains both endpoints of $e_n$, where $L$ is the optimal monotone geodesic $Q$.
\end{rmk}

Now for each $n\geq n_0$, we will inductively construct a simplicial (but not necessarily connected set) $\Gamma_n$ in $X_n$. These sets will have the the following properties for each $n>n_0$:

\begin{enumerate}[(I)]
\item\label{goodedgeprojection} If a point $x_{n-1}\in X_{n-1}$ is contained in the symmetric difference $\Gamma_{n-1} \triangle \pi_{n-1}(\Gamma_n)$, then $x_{n-1}$ is in a good edge $e_{n-1}$ of $X_{n-1}$. Furthermore, there is a unique good edge $f_{n-1}$ in $\Gamma_{n-1}$ with the same endpoints as $e_{n-1}$.
\item \label{largebetabreak} If two distinct edges $e_{n}, e'_{n}$ of $\Gamma_{n}$ are adjacent at a vertex $v_{n}\in \Gamma_{n}$ and have $\pi_0(e_{n})=\pi_0(e'_{n})$ then, for some $k<n$, $v_{n}$ is a vertex of level $k+1$ and $\pi_k(e_{n})=\pi_k(e'_{n})\subset e_k$, where $e_k$ is an edge of $X_k$ with $\beta_E(Q(e_k))\geq\epsilon$.
\item \label{doubleedges} If two edges $e_n, e'_n$ of $\Gamma_n$ share both endpoints, then $\pi_{n-1}(e_n)=\pi_{n-1}(e'_n) \subset e_{n-1}$, where $e_{n-1}$ is an edge of $X_{n-1}$ with $\beta_E(Q(e_{n-1}))\geq \epsilon$.
\end{enumerate}
The inductive hypothesis (\ref{goodedgeprojection}) says that, up to ``double'' edges, $\Gamma_n$ is a lift of $\Gamma_{n-1}$. The hypotheses (\ref{largebetabreak}) and (\ref{doubleedges}) say that any ``non-monotone'' behavior in $\Gamma_n$ can be traced back to an edge at an earlier scale with $\beta\geq \epsilon$.

Suppose now that we have constructed simplicial sets $\Gamma_1, \Gamma_2, \dots, \Gamma_n$ with the above properties. We now construct $\Gamma_{n+1}\subset X_{n+1}$. We do this by addressing each edge of $\Gamma_n$ separately, and using it to define some addition to $\Gamma_{n+1}$.

There are three possible cases for each edge of $\Gamma_n$, which we address in Cases \ref{largebetalift}, \ref{recentlargebetalift}, and \ref{otheredgeslift} below.

\subsection{Edges with large $\beta$}\label{largebetalift}
For each edge $e_n$ of $\Gamma_n$ with $\beta_E(Q(e_n))\geq \epsilon$, add $\pi_{n+1}(S(e_n))$ to $\Gamma_{n+1}$. (Recall the definition of $S(e_n)$ from Definition \ref{Sdefinition}.)

\subsection{Edges with recently large $\beta$}\label{recentlargebetalift}
Suppose an edge $e_n$ of $\Gamma_n$ has $\beta_E(Q(e_n))<\epsilon$, and $\pi_{n-1}(e_n)$ is contained in an edge $e_{n-1}$ with $\beta_E(Q(e_{n-1}))\geq \epsilon$. Then (by \eqref{goodedgeprojection} and the algorithm of \ref{largebetalift}) it must be that $e_n\subseteq \pi_{n}(S(e_{n-1}))$. We then add $\pi_{n+1}(L(e_n))$ to $\Gamma_{n+1}$, where $L(e_n)$ is the portion of $S(e_{n-1})$ that projects onto $e_n$. Note that $L(e_n)$ is a monotone geodesic segment in $X$.

\subsection{The remaining edges}\label{otheredgeslift}
Now consider the edges of $\Gamma_n$ that do not fall into the cases of \ref{largebetalift} or \ref{recentlargebetalift}.

Observe that all such remaining edges $e_n$ of $\Gamma_n$ have $\beta_E(Q(e_n))<\epsilon$. Furthermore, if $e_n$ is in this case, then the edge $e_{n-1}$ of $X_{n-1}$ that contains $\pi_{n-1}(e_n)$ has $\beta_E(Q(e_{n-1}))< \epsilon$.

Let
\begin{equation}\label{Vndef}
V_n = \{v_n\in V(\Gamma_n): \pi_k(v_n) \text{ is a vertex of level } k+1 \text{ in } e_k, \text{ for some } k\leq n \text{ and some edge } e_{k} \text{ in } \Gamma_{k} \text{ with } \beta_E(Q(e_{k}))\geq \epsilon\}.
\end{equation}

Every edge of $\Gamma_n$ falls into some connected component of $\Gamma_n \setminus V_n$. Each such connected component $P_n$ is a simplicial monotone geodesic segment, since by (\ref{largebetabreak}) any vertices that are adjacent to two edges running in the same direction must be in $V_n$. 

If an edge is of the type considered in Case \ref{largebetalift}, then both its endpoints are in $V_n$ and it is its own component. If an edge is of the type considered in Case \ref{recentlargebetalift}, then the the segment $\pi_{n+1}(L(e_n))$ added in Case \ref{recentlargebetalift} has both endpoints in $V_n$ and so is a union of components. For these components, we do nothing, since we have already addressed these edges in the previous cases.

Each remaining edge of $\Gamma_n$ (i.e., one not falling into Cases \ref{largebetalift} or \ref{recentlargebetalift}) lies in some connected component $P_n$ of $\Gamma_n \setminus V_n$ consisting only of this type of edge (those in the current Case \ref{otheredgeslift}). For short, we will call these components $P_n$ of $\Gamma_n\setminus V_n$ ``monotone components''. 

We now perform an inductive ``lifting'' procedure on each such connected monotone component $P_n$ in $\Gamma_n$. This will assign to $P_n$ a (possibly disconnected) monotone set $L(P_n)$ in $X$ such that $\pi_0(L(P_n)) = \pi_0(P_n)$. The set $\pi_{n+1}(L(P_n))$ will be added to $\Gamma_{n+1}$.

Let $P_n$ be such a component in $\Gamma_n$. Observe that, by our inductive assumption \eqref{goodedgeprojection} (and the fact that no edges of $P_n$ are even one above a large $\beta$ edge), there is a unique monotone component $P_{n-1}$ of $\Gamma_{n-1}$ such that $P_n \subset \pi_n(L(P_{n-1}))$.

Order the edges of $P_n$ in monotone order by $e^1_n, e^2_n, \dots, e^r_n$. (Note that $r$ may be $1$.)

Recall Definition \ref{goodedgedef}. If $e^i_n$ is a good edge in $P_n$, define the following
\[ a(e^i_n) = \begin{cases} 
			\text{the left endpoint of } e^i_n & \text{ if } i=1\\
      \text{the first special vertex of level } n+1 \text{ on } e^i_n & \text{ if } i>1 \text{ and } e^{i-1}_n \text{ is bad} \\
     \text{the first special vertex of level } n+2 \text{ on } e^i_n & \text{ if } i>1 \text{ and } e^{i-1}_n \text{ is good}
   \end{cases}
\]

\[ b(e^i_n) = \begin{cases} 
			\text{the right endpoint of } e^i_n & \text{ if } i=r\\
      \text{the last special vertex of level } n+1 \text{ on } e^i_n & \text{ if } i<r \text{ and } e^{i+1}_n \text{ is bad} \\
     \text{the last special vertex of level } n+2 \text{ on } e^i_n & \text{ if } i<r \text{ and } e^{i+1}_n \text{ is good}
   \end{cases}
\]
We call these points ``break points''. Observe that there are at most two break points on any edge, and that all break points are contained in good edges.

Note that $\pi_0(a(e^i_n))\leq \pi_0(b(e^i_n)) \leq \pi_0(a(e^{i+1}_n))$ for each $i$, with possible equality in either or both cases.

We now define the monotone set $L(P_n)$ in $X$. This will be a monotone set such that $\pi_0(L(P_n))=\pi_0(P_n)$. To define it, we need only specify it above a set of intervals covering $\pi_0(P_n))$.

For each good edge $e^i_n$ in $P_n$, $L(P_n)$ is defined above the open interval $(\pi_0(a(e^i_n), \pi_0(b(e^i_n)))$ to be the corresponding segment of the optimal geodesic for $Q(e^i_n)$. If $e^i_n$ and $e^{i+1}_n$ are both good edges, then $L(P_n)$ is defined above $(\pi_0(b(e^i_n)), \pi_0(a(e^{i+1}_n)))$ to be the optimal geodesic for $Q(e^{i+1}_n)$. Finally, for every other point of $\pi_0(P_n)$, $L(P_n)$ above that point is equal to $L(P_{n-1})$ above that point. We then take the closure of this set in $X$ and call it $L(P_n)$.

Figure \ref{fig:constructionpic} gives a pictorial representation of this construction for a given monotone component $P_n\subset \Gamma_n$.

\begin{figure}
	\centering
		\includegraphics[scale=0.4]{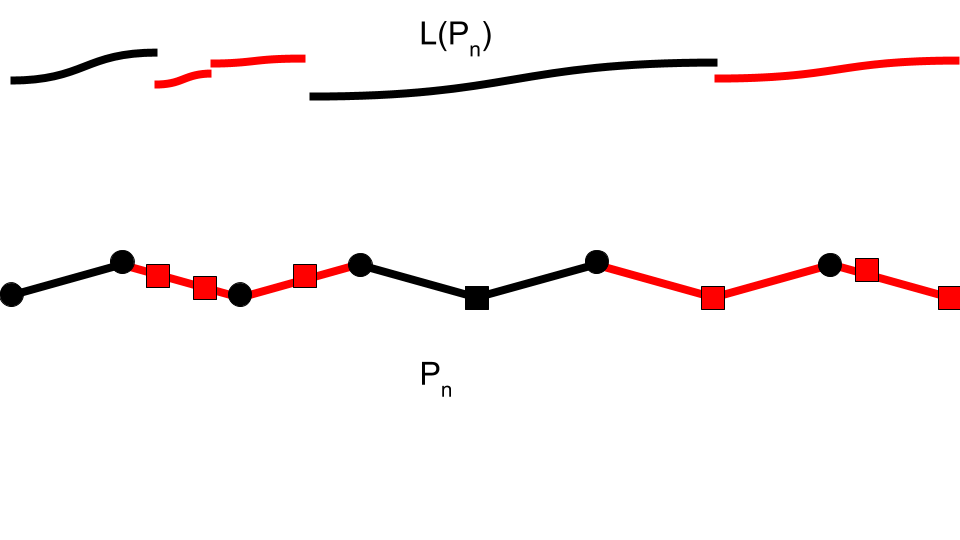}
		\caption{The construction of $L(P_n)\subset X$ given $P_n\subset \Gamma_n$. The red edges represent good edges of $P_n$ while the black edges represent bad edges of $P_n$. The red squares on $P_n$ represent the break points $a(e^i_n)$ and $b(e^i_n)$. A black square represents a possible special vertex on a bad edge of $P_n$. The red segments in $L(P_n)$ represent the portions on which $L(P_n)$ is defined as part of the optimal monotone geodesic segment for some $Q(e^i_n)$, while the black segments in $L(P_n)$ represent the portions on which $L(P_n)$ is defined to be $L(P_{n-1})$.}
	\label{fig:constructionpic}
\end{figure}

Finally, we add $\pi_{n+1}(L(P_n))$ to $\Gamma_{n+1}$, for each monotone component $P_n$ in $\Gamma_n$ as defined above.

This completes the definition of the set $\Gamma_{n+1}$ in $X_{n+1}$.

\subsection{Verifying the inductive hypotheses}
We must now verify our inductive assumptions, namely that $\Gamma_{n+1}$ is simplicial and properties (\ref{goodedgeprojection}), (\ref{largebetabreak}), and (\ref{doubleedges}) hold for $\Gamma_{n+1}$.

Observe that each edge $e_{n+1}$ in $\Gamma_{n+1}$ was added to $\Gamma_{n+1}$ by the above algorithm in exactly one way, namely either in \ref{largebetalift}, \ref{recentlargebetalift}, or \ref{otheredgeslift}.

\begin{claim}\label{simpliciallift}
$\Gamma_{n+1}$ is simplicial in $X_{n+1}$. 

Furthermore, If $e_n^i, e_n^{i+1}, \dots, e_n^j$ are consecutive good edges in a monotone component $P_n$ of $\Gamma_n$, then the portion of $\pi_{n+2}(L(P_n))$ that projects onto $(a(e^i_n), b(e^j_n))\subset X_0 \cong [0,1]$ is connected
\end{claim}
\begin{proof}
Let us first prove the second part of the claim. Consider the portion of $L(P_n)$ projecting onto
$$(\pi_0(a(e^i_n)),\pi_0( b(e^j_n))).$$
It is a monotone set that is a finite union of monotone geodesic segments that project onto consecutive adjacent intervals in $X_0$. By the algorithm in \ref{otheredgeslift} (and Lemma \ref{nearestpoint}), the ``vertical'' gap in $X$ between any two consecutive monotone geodesic segments in this span is at most
\begin{equation}\label{gap1}
2(4\epsilon A m^{-n} + \frac{1}{2}m^{-(n+2)}) < 2m^{-(n+2)}.
\end{equation}
and occurs at a vertex of level $n+2$. (The two optimal monotone segments we splice together above a break point $\pi_0(v'_n)$ are within distance $4\epsilon A m^{-n}$ above some point $\pi_0(x)$, where $x\in E$ and $|\pi_0(x)-\pi_0(v'_n)|\leq\frac{1}{2}m^{-(n+2)}$, since $v'_n$ is a special vertex of level $n+2$. The two monotone geodesics can thus diverge from each other only a further distance of $2(\frac{1}{2}m^{-(n+2)})$ above the point $\pi_0(v'_n)$.) 

Any two distinct vertices of level $n+2$ in $X_{n+2}$ with the same image under $\pi_0$ have distance at least $2m^{-(n+2)}$. It follows that the portion of $\pi_{n+2}(L(P_n))$ projecting onto $(a(e^i_n), b(e^j_n))$ is connected, which proves the second part of the claim.

Note that this also implies that the portion of $\pi_{n+1}(L(P_n))$ between $a(e^i_n)$ and $b(e^{j}_n)$ is connected. For the first part of the claim, note that the only way that $\Gamma_{n+1}$ could fail to be simplicial was if two adjacent good edges in $\Gamma_n$ ``broke apart'' at a vertex of level $n+2$ when lifted to $\Gamma_{n+1}$. But by the second part of the claim, proven above, this cannot occur.
\end{proof}

\begin{claim}\label{edgeslift}
If a point $x_n\in X_{n}$
is contained in $\Gamma_{n} \triangle \pi_{n}(\Gamma_{n+1})$, then $x_n$ is in a good edge $e_{n}$ of $X_{n}$.

Furthermore, there is a unique good edge $f_{n}$ in $\Gamma_{n}$ with the same endpoints as $e_{n}$.

In particular, the inductive hypothesis \eqref{goodedgeprojection} holds for $\Gamma_{n+1}$.
\end{claim}
\begin{proof}
Suppose $x_n\in \Gamma_{n} \setminus \pi_{n}(\Gamma_{n+1})$. Then, by construction, the only possibilities are: $x_n$ lies between $a(e^i_n)$ and $b(e^i_n)$ for some good edge $e^i_n$ in $\Gamma_n$, or $x_n$ lies between $b(e^i_n)$ and $a(e^{i+1}_n)$ for some pair of adjacent good edges $e^i_n$ and $e^{i+1}_n$ in $\Gamma_n$, or possibly $x_n$ lies between $a(e^i_n)$ and $b(e^{i+2}_n)$ for some triple of adjacent consecutive good edges $e^i_n, e^{i+1}_n, e^{i+2}_n$ in $\Gamma_n$. In any case, $x_n$ lies on a good edge of $\Gamma_n$ itself.

On the other hand, suppose now that $x_n\in \pi_{n}(\Gamma_{n+1}) \setminus \Gamma_n$. Let $x_{n+1}$ be a point of $\Gamma_{n+1}$ such that $\pi_n(x_{n+1})=x_n$. Then $x_{n+1}$ must have been added to $\Gamma_{n+1}$ in Case \ref{otheredgeslift}. Furthermore, there must be a string of consecutive good edges $e^i_n, e^{i+1}_n, \dots, e^j_n$ (possibly with $j=1$) in $\Gamma_n$ such that
$$\pi_0(x_n) \in [\pi_0(a(e^i_n)), \pi_0(b(e^j_n))]$$
and $x_{n+1}\in \pi_{n+1}(L)$, where $L$ is the optimal geodesic for some $e^k_n$ that is within two edges of $x_n$. It follows that $x_n\in \pi_n(L)$, and therefore, by Remark \ref{goodedgeremark}, $x_n$ is contained in an edge that shares both endpoints with one of the good edges $e^\ell_n$.

Now note that if there were two such good edges $f_n$ in $\Gamma_n$, then by (\ref{doubleedges}) for $\Gamma_n$, they would both have been added in the algorithm of \ref{largebetalift} applied to $\Gamma_{n-1}$. Hence, they (and any edge of $\Gamma_n$ sharing both endpoints with them) would both simply be lifted as in \ref{recentlargebetalift} applied to $\Gamma_n$, and hence the situation of the Claim could not occur.
\end{proof}

\begin{claim}\label{largebetabreakclaim}
Suppose that two distinct edges $e_{n+1}, e'_{n+1}$ of $\Gamma_{n+1}$ are adjacent at a vertex $v_{n+1}\in \Gamma_{n+1}$ and have $\pi_0(e_{n+1})=\pi_0(e'_{n+1})$.

Then, for some $k<n+1$, $v_{n+1}$ is a vertex of level $k+1$ and $\pi_k(e_{n+1})=\pi_k(e'_{n+1})\subset e_k$, where $e_k$ is an edge of $X_k$ with $\beta_E(Q(e_k))\geq\epsilon$.

In particular, the inductive hypothesis (\ref{largebetabreak}) holds for $\Gamma_{n+1}$.
\end{claim}
\begin{proof}
Let $e_n$ and $e'_n$ be the (not necessarily distinct) edges of $X_{n}$ containing $\pi_n(e_{n+1})$ and $\pi_n(e'_{n+1})$, respectively.

By Claim \ref{edgeslift}, $e_n$ and $e'_n$ share both endpoints with (not necessarily distinct) edges $f_n$ and $f'_n$ of $\Gamma_n$, respectively. 

Suppose first that $f_n$ and $f'_n$ can be chosen so that $f_n\neq f'_n$. Observe that in this case it must be that $f_n$ and $f'_n$ share at least one vertex $v_n=\pi_n(v_{n+1})$. It then follows from (\ref{largebetabreak}) in $\Gamma_{n}$ that, for some $k<n$, $v_{n}$ is a vertex of level $k+1$ and $\pi_k(f_{n})=\pi_k(f'_{n})\subset e_k$, where $e_k$ is an edge of $X_k$ with $\beta_E(Q(e_k))\geq\epsilon$. Since $k<n$, we have $\pi_k(f_n)=\pi_k(e_n)$ and $\pi_k(f'_n)=\pi_k(e'_n)$, and therefore the result follows.

Now suppose that we cannot choose $f_n$ and $f'_n$ to be distinct. Then $e_n$ and $e'_n$ both share both their endpoints with exactly one edge $f_n$ of $\Gamma_n$.

If $e_n=e'_n$, then, by inspecting the algorithm, the only possibility is that $e_{n+1}$ and $e'_{n+1}$ were both added to $\Gamma_{n+1}$ as in \ref{largebetalift}, and hence the conclusion of the claim is satisfied with $k=n$.

If $e_n\neq e'_n$, then by Claim \ref{edgeslift}, $f_n$ is a good edge of $\Gamma_n$, and both edges $e_{n+1}$ and $e'_{n+1}$ were added to $\Gamma_{n+1}$ using the algorithm in \ref{otheredgeslift} applied to the montone component containing $f_n$. But this is impossible, since the algorithm in \ref{otheredgeslift} generates monotone sets and $e_{n+1}$ and $e'_{n+1}$ cannot be in the same monotone set.
\end{proof}

\begin{claim}
If $e_{n+1}$ and $e'_{n+1}$ are edges of $\Gamma_{n+1}$ that share both endpoints, then the edge $e_n$ of $X_n$ containing $\pi_n(e_{n+1})=\pi_n(e'_{n+1})$ has $\beta_E(Q(e_n))\geq \epsilon$.

In particular, the inductive hypothesis (\ref{doubleedges}) holds for $\Gamma_{n+1}$.
\end{claim}
\begin{proof}
Let $v_{n+1}$ and $w_{n+1}$ be the shared endpoints of $e_{n+1}$ and $e'_{n+1}$.

Claim \ref{largebetabreakclaim} says that for some $k<n+1$, both $v_{n+1}$ and $w_{n+1}$ are vertices of level $k+1$, and $\pi_k(e_{n+1})=\pi_k(e'_{n+1})\subset e_k$, where $e_k$ is an edge of $X_k$ with $\beta_E(Q(e_k))\geq\epsilon$.

Since $v_{n+1}$ and $w_{n+1}$ are adjacent in $X_{n+1}$, the only way that they can both be vertices of level $k+1$ is if $k+1=n+1$. Therefore $k=n$ and the conclusion of the claim follows.
\end{proof}

This completes the construction of the sets $\Gamma_n$.

It is immediate from the lifting algorithm in Cases \eqref{largebetalift}, \eqref{recentlargebetalift}, and \eqref{otheredgeslift} that, for each $n\geq n_0$,
$$ |\Gamma_n| \leq m^{-n_0} + C_X \sum_{Q\in\mathcal{Q}: \beta_E(Q)\geq\epsilon} \diam(Q).$$
Indeed, Case \eqref{largebetalift} is the only case in which $|\Gamma_n|$ increases from $|\Gamma_{n-1}|$.

\section{Additional lemmas about the algorithm}\label{algorithmlemmas}

\begin{lemma}\label{projection}
For any $n>n_0$,
$$\pi_{n-1}(\Gamma_{n+1}) = \pi_{n-1}(\Gamma_n).$$
\end{lemma}
\begin{proof}
This follows immediately from Claim \ref{edgeslift} and the following simple observation: If two edges $e_n$ and $f_n$ in $X_n$ share the same endpoints, then $\pi_{n-1}(e_n) = \pi_{n-1}(f_n)$.
\end{proof}

\begin{lemma}\label{badedgeslift}
Let $e_n$ be a bad edge of $\Gamma_n$, and let $\ell\geq n$. Then $e_n \subset \pi_n(\Gamma_\ell)$.
\end{lemma}
\begin{proof}
We induct on $\ell-n$. If $\ell=n+1$, the conclusion of the lemma follows from Claim \ref{edgeslift}. If $\ell>n+1$, then by Lemma \ref{projection} and the inductive hypothesis,
$$ \pi_n(\Gamma_\ell) = \pi_n \pi_{\ell-2} \Gamma_\ell = \pi_n \pi_{\ell-2} \Gamma_{\ell-1} = \pi_n \Gamma_{\ell-1} \supset e_n.$$
\end{proof}

\begin{lemma}\label{onetoone}
Let $n_0\leq i < n$ and let $e_n$ and $e'_n$ be edges in $\Gamma_n\subset X_n$ such that $\pi_i(e_n)=\pi_i(e'_n)$.

Suppose that, for each $i\leq k < n$, the edge $e_k\subset X_k$ containing $\pi_k(e_n)$ has $\beta_E(Q(e_k))<\epsilon$, and that the same holds for $e'_n$.

Then $e_n = e'_n$.
\end{lemma}
\begin{proof}
We prove this by induction on $n-i$. Suppose first that $n=i+1$, so that $\pi_{n-1}(e'_n) = \pi_{n-1}(e_n) \subset e_{n-1}$. Note that, by assumption, $\beta_E(Q(e_{n-1}))<\epsilon$.

If $\pi_{n-2}(e_{n-1})$ lies in an edge $e_{n-2}$ with $\beta_E(Q(e_{n-2}))\geq \epsilon$, then $e_{n-1}\subset \Gamma_{n-1}$ and $e_n=e'_n$, as then $e_{n-1}$ would have been lifted in a one-to-one way in \ref{recentlargebetalift}.

Otherwise, by (\ref{goodedgeprojection}) and (\ref{doubleedges}) we have that there is a unique edge $f_{n-1}$ in $\Gamma_{n-1}$ with the same endpoints as $e_{n-1}$, and therefore having $\beta_E(Q(f_{n-1}))<\epsilon$. Thus, both edges $e_n$ and $e'_n$ must have been added to $\Gamma_n$ by applying the algorithm of \ref{otheredgeslift} to the monotone component of $f_n$, and so we must have $e_n=e'_n$.

Now suppose that $n>i+1$. Let $e_{n-1}$ and $e'_{n-1}$ be edges of $X_{n-1}$ containing $\pi_{n-1}(e_n)$ and $\pi_{n-1}(e'_n)$, respectively. Then there are edges $f_{n-1}$ and $f'_{n-1}$ of $\Gamma_{n-1}$ sharing both endpoints with $e_{n-1}$ and $e'_{n-1}$, respectively. Then $f_{n-1}$ and $f'_{n-1}$ have the same property as $e_n$ and $e'_n$ in $\Gamma_{n-1}$. Therefore, by induction, $f_{n-1}=f'_{n-1}$, and so $e_{n-1}$ and $e'_{n-1}$ share both endpoints.

Since $n>i+1$, the assumptions imply that $e_n$ and $e'_n$ must have been added to $\Gamma_n$ in case \ref{otheredgeslift}, applied to the component containing $f_{n-1}$. It follows that $e_n=e'_n$.

\end{proof}

\begin{lemma}\label{projectdata}
Let $e_n$ be an edge of $\Gamma_n$. Suppose there is a point $x\in E$ such that
$$ \dist(\pi_n(x), e_n \cap \pi_0^{-1}(\pi_0(x))) < m^{-n}. $$
Let $e_{n-1}$ be the edge of $X_{n-1}$ containing $\pi_{n-1}(e_n)$, and suppose that $\beta_E(Q(e_{n-1}))<\epsilon$. Then one of the following possibilities must occur:
\begin{enumerate}[(i)]
\item\label{projectdata1} $e_{n-1}$ is a bad edge of $\Gamma_{n-1}$, and $\pi_{n-1}(e_n)$ contains an endpoint of $e_{n-1}$.
\item\label{projectdata2} $e_{n-1}$ is an edge of $\Gamma_{n-1}$ of the type in \ref{recentlargebetalift}.
\item\label{projectdata3} $e_{n-1}$ shares both endpoints with a good edge $f_{n-1}$ of $\Gamma_{n-1}$, and $e_n\subset \pi_n(L)$, where $L$ is the optimal monotone geodesic of $Q(f_{n-1})$ 
\item\label{projectdata4} $e_{n-1}$ shares both endpoints with a good edge $f_{n-1}$ of $\Gamma_{n-1}$, and $e_{n}\subset \pi_n(L) \cup \pi_n(L')$, where $L,L'$ are the optimal monotone geodesics for good edges in $\Gamma_{n-1}$ within two edges of $f_{n-1}$; furthermore, in this case, $\pi_{n+2}(L \cup L')$ is connected.
\item\label{projectdata5} $e_{n-1}$ is a good edge of $\Gamma_{n-1}$, and $\pi_{n-1}(e_n)$ contains either $a(e_{n-1})$ or $b(e_{n-1})$, and there is a bad edge $f_{n-1}$ of $\Gamma_{n-1}$ adjacent to $e_{n-1}$ with $\beta_E(Q(f_{n-1}))<\epsilon$. 
\end{enumerate}
\end{lemma}
\begin{proof}
By (\ref{goodedgeprojection}), $e_{n-1}$ must either be a bad edge inside $\Gamma_{n-1}$ or share both endpoints with a good edge of $\Gamma_{n-1}$.

If $e_{n-1}$ is a bad edge of $\Gamma_{n-1}$,  then $\pi_{n-1}(e_n)$ must contain an endpoint of $e_{n-1}$, and so \eqref{projectdata1} holds. Indeed, suppose not. There is a point $x\in E$ with $\pi_0(x)\in\pi_0(e_n)\subset\pi_0(e_{n-1})$ such that 
$$ \dist(\pi_{n-1}(x), e_{n-1} \cap \pi_0^{-1}(\pi_0(x))) \leq \dist(\pi_n(x), e_n \cap \pi_0^{-1}(\pi_0(x))) < m^{-n}.$$
If $\pi_{n-1}(e_n)$ does not contain an endpoint of $e_{n-1}$, then $\pi_{n-1}(x)$ must be in $\pi_{n-1}(e_n) \subset e_{n-1}$, therefore within $(1-2A\epsilon)m^{-(n-1)}$ of both endpoints of $e_{n-1}$, and this violates the assumption that $e_{n-1}$ is bad.

Now suppose that $e_{n-1}$ shares both endpoints with a good edge $f_{n-1}$ of $\Gamma_{n-1}$. Suppose that none of \eqref{projectdata2}, \eqref{projectdata3}, or \eqref{projectdata4} hold. In that case, by (\ref{goodedgeprojection}) and the algorithm of \ref{otheredgeslift}, we must have that $e_{n-1}=f_{n-1}$ and that $\pi_{n-1}(e_n)$ projects to the ``left'' of $a(e_{n-1})$ or to the ``right'' of $b(e_{n-1})$. In either case, the adjacent edge in that direction must be a bad edge of $\Gamma_{n-1}$ (otherwise we would be in either case \eqref{projectdata3} or \eqref{projectdata4}), and $\pi_{n-1}(e_n)$ must contain either $a(e_{n-1})$ or $b(e_{n-1})$ by the definition of these points. In other words, \eqref{projectdata5} holds.

The fact that, if \eqref{projectdata4} holds, then $\pi_{n+2}(L \cup L')$ is connected follows from the same argument as in Claim \ref{simpliciallift}, equation \eqref{gap1}.
\end{proof}

\section{Connectability}\label{connectability}
In this section, we prove that we can make the sets $\Gamma_n\subset X_n$ that we constructed in Section \ref{liftingalgorithm} into connected sets, by adding suitable connections with controlled length. The length of such an added connection may be controlled in one of two ways: either by $\diam(Q)$ for some $Q\in\mathcal{Q}$ with $\beta_E(Q)\geq \epsilon$, or by some fraction of the length of $\Gamma_n$ itself. Of course, we must be careful to show that there is no overcounting.

It will be convenient to introduce the notion of an $N$-overlapping collection of subsets: A collection $\mathcal{T}$ of subsets of some $X_n$ is called \textit{$N$-overlapping}, for some constant $N$, if no point of $X_n$ is contained in more than $N$ elements of $\mathcal{T}$.

The remainder of this section is devoted to the proof of of the following lemma.

\begin{lemma}\label{connectable}
There are constants $N_{\eta,\Delta}$, $C_{\eta,\Delta}$,  and $C_X$, such that for each $n\geq n_0$, there are
\begin{itemize}
\item two finite collections $\mathcal{C}^1_n$ and $\mathcal{C}^2_n$ consisting of simplicial geodesic arcs $C:[0,1]\rightarrow X_n$,
\item an $N_{\eta}$-overlapping collection $\mathcal{T}_n$ of subsets of $\Gamma_n$ that are each disjoint from $\pi_n(E)$,
\item and maps $c_n^1:\mathcal{C}^1_n\rightarrow \mathcal{T}_n$ and $c_n^2:\mathcal{C}_n^2\rightarrow \mathcal{Q}$,
\end{itemize}
with the property that $\Gamma_n \cup \bigcup_{C\in\mathcal{C}_n^1 \cup \mathcal{C}_n^2} C$ is connected,
\begin{equation}\label{connectable1}
\sum_{C\in \mathcal{C}_n^1, c^1(C) = T} |C| \leq C_{\eta,\Delta} m^{-1}|T|
\end{equation}
for each $T\in \mathcal{T}_n$, and
\begin{equation}\label{connectable2}
 \sum_{C\in \mathcal{C}_n^2, c^2(C) = Q} |C| \leq C_{X} \diam(Q)
\end{equation}
for each $Q\in \mathcal{Q}$.

Furthermore, we have
\begin{equation}\label{connectable3}
 \beta_E(c_n^2(C)) \geq \epsilon
\end{equation}
for each $n\geq 0$ and each $C\in \mathcal{C}_n^2$.

\end{lemma}

For a curve $C$ in $\mathcal{C}^1$ or $\mathcal{C}^2$, we freely identify the parametrized curve $C\colon [0,1]\rightarrow X_n$ with its image $C([0,1])$. Such an arc $C$ has two endpoints, namely $C(0)$ and $C(1)$, which are vertices of $X_n$.

The entirety of this section is devoted to the proof of Lemma \ref{connectable}. The proof is by induction on $n$. For $n=n_0$, the set $\Gamma_{n_0}$ is simply the edge $e_{n_0}$ and thus connected, so we set $\mathcal{C}^1_{n_0} = \mathcal{C}^2_{n_0} = \emptyset.$

Suppose now that $\mathcal{C}^1_{n}$ and $\mathcal{C}^2_{n}$ have been constructed. We now construct $\mathcal{C}^1_{n+1}$ and $\mathcal{C}^2_{n+1}$. To do so, we consider only cases where a connected component in $\Gamma_n$ may ``break'' when lifted to $\Gamma_{n+1}$ by the algorithm of Section \ref{liftingalgorithm}. This may happen at any vertex of $V_n$ (which was defined in \eqref{Vndef}), or it may happen while applying the algorithm of \ref{otheredgeslift} within a monotone component $P_n$ of $\Gamma_n$.

First, consider all vertices $v_n\in V_n$ of $\Gamma_n$, where $V_n$ was defined in \ref{otheredgeslift}. 
Let $i\leq n$ be the largest natural number such that $\pi_i(v_n)$ lies in an edge $e_i$ of $X_i$ with $\beta_E(Q(e_i))\geq \epsilon$. Then in $\Gamma_{n+1}$, add connections $C^2_{n+1}$ to $\mathcal{C}^2_{n+1}$ between all lifts of $v_n$ lying in $\Gamma_{n+1}$. Set $c_{n+1}^2(C^2_{n+1})=Q(e_i)$ for each of these connections. 
Observe that the total length of these connections for each $v_n$ is at most $C_X m^{-n}$. 

Now consider any monotone component $P_n$ of $\Gamma_n$ as in Case \ref{otheredgeslift}. Let $v'_n$ be a vertex of level $n+1$ in $P_n$ which has two different ``lifts'' $v_{n+1}$ and $w_{n+1}$ in $V(\Gamma_{n+1})$. (We have already argued in Claim \ref{simpliciallift} that this cannot occur if $v'_n$ is a vertex of level $n+2$.) 

In other words, in the language of \ref{otheredgeslift}, the monotone set $\pi_{n+1}(L(P_n))$ contains two different vertices above $\pi_0(v'_n)$, namely $v_{n+1}$ and $w_{n+1}$.

We need to join $v_{n+1}$ and $w_{n+1}$ by adding a connection to $\mathcal{C}^1_{n+1}$ or $\mathcal{C}^2_{n+1}$. To do so, we need to examine the various cases in which this ``break'' could have occurred.  Observe that $d_{X_{n+1}}(v_{n+1},w_{n+1})\leq C_\eta m^{-n}$, as both project onto $v'_n$ or an adjacent edge by (\ref{goodedgeprojection}). Furthermore, since $v_{n+1}$ and $w_{n+1}$ are distinct vertices of $X_{n+1}$ and have the same projection to $X_0$, it must be that $d_{X_{n+1}}(v_{n+1}, w_{n+1}) \geq 2m^{-(n+1)}$.

\textbf{(Case A)} The point $v'_n$ is equal to $a(e^i_n)$ for some good edge $e^i_n$ of $P_n$ such that $e^{i-1}_n$ is in $P_n$ and is bad. (Or, similarly, $v'_n$ is equal to $b(e^i_n)$ for some good edge $e^i_n$ of $P_n$ such that $e^{i+1}_n$ is in $P_n$ and is bad. We just handle the first possibility, since the second is identical.)

We first claim that a neighborhood of $v'_n$ in $P_n$ cannot be part of $\pi_n(L \cap Q(e_{n-1}))$ where $L$ is the optimal geodesic for $Q(e_{n-1})$, for some good edge $e_{n-1}$ of $\Gamma_{n-1}$. Indeed, suppose it were. Then to the left of $v'_n$, $L(P_n)$ would be equal to $L$, and to the right of $v'_n$, $L(P_n)$ would be equal to the optimal geodesic for $Q(e^i_n)$, $Q(e^{i+1}_n)$, or $Q(e^{i+2}_n)$. In particular, by a similar argument as in \eqref{gap1} the size of the ``gap'' at $v'_n$ between these two geodesics, and hence between $v_{n+1}$ and $w_{n+1}$ is at most
\begin{equation}\label{connectabilitygap}
 (4A\epsilon m^{-n} + \frac{1}{2}m^{-(n+1)}) + (4A\epsilon m^{-(n-1)} + \frac{1}{2}m^{-(n+1)}) < 2m^{-(n+1)},
\end{equation}
contradicting our assumption that $v_{n+1}\neq w_{n+1}$. 

Therefore, as in the proof of Claim \ref{edgeslift}, it must be that $\pi_{n-1}(v'_n)$ is in $\Gamma_{n-1}$. Let $e_n$ be an edge of $P_n$ containing $v'_n$ such that 
\begin{equation}\label{env'n}
\dist(e_n, \pi_0^{-1}(\pi_0(x))) < m^{-n}
\end{equation}
for some $x\in E$, which must exist as $v'_n$ is a special vertex on $P_n$.

Consider $\pi_{n-1}(e_n)$, which must be in an edge of $\Gamma_{n-1}$ by Claim \ref{edgeslift}.

\textbf{(Case A1)}
Suppose first that the edge of $\Gamma_{n-1}$ containing $\pi_{n-1}(e_n)$ is an edge as in \ref{recentlargebetalift}. Then there is an edge $e_{n-2}$ in $\Gamma_{n-2}$ with $\beta_E(Q(e_{n-2}))\geq\epsilon$. We add a connection $C_{n+1}$ to $\mathcal{C}^2_{n+1}$ connecting $v_{n+1}$ and $w_{n+1}$, with $|C_{n+1}|\leq C_\eta m^{-n}$. We then set $c^2_{n+1}(C_{n+1}))=Q(e_{n-2})$.

\textbf{(Case A2)}
Suppose now that the edge of $\Gamma_{n-1}$ containing $\pi_{n-1}(e_n)$ is not an edge as in \ref{recentlargebetalift}, and is therefore in a monotone component $P_{n-1}$ as in \ref{otheredgeslift}. It follows from the discussion leading to \eqref{connectabilitygap} above that this edge of $\Gamma_{n-1}$ is either bad itself, or adjacent to a bad edge  of $P_{n-1}$.

In either case, there is a bad edge $e_{n-1}$ in $\Gamma_{n-1}$ within distance $m^{-(n-1)}$ from $\pi_{n-1}(v'_n)$ that has $\beta_E(Q(e_{n-1}))<\epsilon$. Therefore, there is an arc $I_{n-1} \subset e_{n-1}$ such that $|I_{n-1}| = m^{-(n-1)}/10$ and
$$m^{-(n-1)}/10 \leq \dist(I_{n-1}, \pi_{n-1}(E)) \leq 2m^{-(n-1)}.$$
That arc lifts to a (possibly not connected) set $T_{n+1}$ in $\Gamma_{n+1}$, by Lemma \ref{badedgeslift}.

We add $T_{n+1}$ to $\mathcal{T}_{n+1}$. We also add a connection $C_{n+1}$ to $\mathcal{C}^1_{n+1}$ that joins $v_{n+1}$ to $w_{n+1}$, and set $c_1(C_{n+1})=T_{n+1}$. Note that $|C_{n+1}|=C_\eta m^{-n} \leq C_\eta m^{-1} |T_{n+1}|$.

\textbf{(Case B)} The point $v'_n$ is not in $[a(e^i_n), b(e^i_n)]$ or $[b(e^i_n), a(e^{i+1}_n)]$ for any good edges $e^i_n, e^{i+1}_n$ in $P_n\subset \Gamma_n$.

In other words, in a neighborhood of $v'_n$ we are lifting according to $L(P_{n-1})$, where $P_{n-1}$ is a monotone component of $\Gamma_{n-1}$, and $\pi_{n+1}(L(P_{n-1}))$ ``breaks'' at this point.

Let $j\geq 1$ be the smallest integer such that $\pi_{n-j}(v'_n)$ is in $[a(e^i_{n-j}), b(e^i_{n-j})]$ or $[b(e^i_{n-j}), a(e^{i+1}_{n-j})]$ for some good edge or edges $e^i_{n-j}$ and $e^{i+1}_{n-j}$ in $\Gamma_{n-j}$ (or, if $j=1$, an edge that shares both endpoints with such a good edge). Note that this must have occured for some $j\geq 1$, otherwise there could be no break here when lifting at $v'_n$.

If $j=1$, then it must be that $\pi_{n-1}(v'_n)$ is either $a(e^i_{n-1})$ or $b(e^i_{n-1})$ for a good edge $e^i_{n-1}$. Suppose the former; the latter case is identical. In that case, it must be that $e^{i-1}_{n-1}$ is bad in $\Gamma_{n-1}$, otherwise there would be no break. (Indeed, similar to \eqref{gap1}, we would have a gap of size at most
$$ 2\left(\frac{1}{2}m^{-(n+1)} + 4A\epsilon m^{-(n-1)}\right) < 2m^{-(n+1)},$$
and so $v'_n$ would have a unique lift in $X_{n+1}$.)
Therefore if $j=1$, $e^{i-1}_{n-1}$ must be a bad edge. We then run the exact same argument as in Case A2, adding the lift $T_{n+1}$ of an arc $I_{n-1}\subset \Gamma_{n-1}$ to $\mathcal{T}_{n+1}$, and defining a connection $C_{n+1}\in \mathcal{C}^1_{n+1}$ between the two lifts of $v'_n$ with $c^1_{n+1}(C_{n+1})=T_{n+1}$.

If $j>1$, we look at scale $n-j+1$. By definition of $j$, $\pi_{n-j}(v'_n)$ lies in a good edge of $\Gamma_{n-j}$, but $\pi_{n-j+1}(v'_n)$ lies in or adjacent to an edge $e_{n-j+1}$ of $\Gamma_{n-j+1}$ which is bad and has $\beta_E(Q(e_{n-j+1}))<\epsilon$. There is therefore an arc $I_{n-j+1} \subset e_{n-j+1}$ such that $|I_{n-j+1}| = m^{n-j+1}/10$ and
$$\dist(I_{n-j+1}, \pi_{n-j+1}(E))\geq m^{-(n-j+1)}/10.$$
Furthermore, since $\pi_{n-j}(e_n)$ is in or adjacent to a good edge, we must have that
$$\dist(I_{n-j+1}, \pi_{n-j+1}(E))\leq C_{\eta}m m^{-(n-j+1)}.$$
We lift this arc to a (possibly disconnected) set $T_{n+1}$ in $\Gamma_{n+1}$, add $T_{n+1}$ to $\mathcal{T}_{n+1}$, and define a connection $C_{n+1}\in\mathcal{C}^1_{n+1}$ between $v_{n+1}$ and $w_{n+1}$ with $c_1(C_{n+1})=T_{n+1}$ and $|C_{n+1}|\leq C_\eta m^{-n}$.

Observe also that in Case (B), it must be that, for each $n-j \leq k < n$, $\pi_{k}(v'_n)$ is not contained in an edge $e_k$ with $\beta_E(Q(e_k))\geq \epsilon$. Indeed, if it were, then a neighborhood of $v'_n$ would simply be lifted according to a connected component of $S(e_k)$, as a connected set, and there would be no break here.

Finally, we also add to $\mathcal{C}^1_{n+1}$, $\mathcal{C}^2_{n+1}$, and $\mathcal{T}_{n+1}$ appropriate ``lifts'' of the elements of $\mathcal{C}^i_{n}$ and $\mathcal{T}_n$. Namely, by Lemma \ref{badedgeslift}, each element $T_n$ of $\mathcal{T}_n$ has a corresponding (not necessarily connected) lift to $\Gamma_{n+1}$, of the same length. Add that lift to $\mathcal{T}_{n+1}$.

For every $C_n\in \mathcal{C}^1_n$ such that $c_n(C_n) = T_n$, the vertices $C_n(0)$ and $C_n(1)$ lift (possibly non-uniquely) to vertices in $\Gamma_{n+1}$. Choose a lift of each endpoint of $C_n$ in $\Gamma_{n+1}$, and add a geodesic arc $C_{n+1}$ to $\mathcal{C}^1_{n+1}$ joining those two lifts. Each new connection $C_{n+1}$ defined in this way is longer than the corresponding $C_n$ by at most $C_\eta m^{-n}$. We then set $c^1_{n+1}(C_{n+1})$ to be the lift of $T_n$.

For each $C_n\in\mathcal{C}^2_n$, we similarly add an arc $C_{n+1}$ to $\mathcal{C}^2_{n+1}$ whose endpoints are any lifts in $\Gamma_{n+1}$ of those of $C_n$, and we set $c^2_{n+1}(C_{n+1}) = c^2_n(C_n) \in \mathcal{Q}$.

Note that if a connection $C_{n+1}$ was originally added to $\mathcal{C}^2_{n+1}$ or $\mathcal{C}^2_{n+1}$ by one of the above cases, then the length of any of its lifts to higher scales in this way is bounded by $C_\eta m^{-n}$.

This completes the definition of the collections $\mathcal{T}_{n+1}, \mathcal{C}^1_{n+1}, \mathcal{C}^2_{n+1}$ and the maps $c^1_{n+1}$ and $c^2_{n+1}$. We now must verify the properties in Lemma \ref{connectability}, for each $n\geq n_0$.

The fact that 
$$\Gamma_{n+1} \cup \bigcup_{C\in\mathcal{C}_{n+1}^1 \cup \mathcal{C}_{n+1}^2} C$$
is connected follows by induction. Indeed, the only way that it could fail to be connected is if 
\begin{itemize}
\item a vertex $v_n\in V_n$ lifts to two different vertices in $\Gamma_{n+1}$ (which can happen in Cases \ref{largebetalift} or \ref{otheredgeslift} of the lifting algorithm), or
\item for some monotone component $P_n\subset \Gamma_n$, $\pi_{n+1}(L(P_n))$ is not connected
\end{itemize}
and in both of these cases we have added connections between the relevant lifts.

Now we show the other properties stated in the lemma. 
\begin{claim}
The collection $\mathcal{T}_n$ is $N_{\eta}$-overlapping, for a constant $N_\eta$ that is bounded above in terms of $\eta$.
\end{claim}
\begin{proof}
We will show that each point of $\Gamma_n$ is contained in at most $N_{\eta}=1+\log(20C_{\eta} m)/\log(m)$ elements of $\mathcal{T}_n$. The constant $N_\eta$ is bounded above in terms of $\eta$, by $2+\frac{\log(20C_\eta)}{\log(2)}$.

Suppose $T$ and $T'$ are in $\mathcal{T}_n$ and distinct and intersect. Then $T$ and $T'$ are lifts of arcs $I_i\subset \Gamma_i\subset X_i$ and $I_{i'}\subset \Gamma_{i'} \subset X_{i'}$, for some $i,i'<n$, as above.

Note first that we must have $i\neq i'$: distinct arcs formed in $X_i$ are disjoint, as they are subsets of different edges.

Now, by construction, we have that
$$ m^{-i}/10 \leq \dist(T,\pi_n(E)) \leq (C_{\eta} m )m^{-i} \text{ and } m^{-i}/10 \leq \diam(T) \leq (C_\eta  m) m^{-i},$$
and the analogous statement holds for $T'$ and $i'$.

It follows that
$$ |i-i'| \leq 1+\log(20C_{\eta})/\log(m) = N_{\eta}. $$

Thus, each $x\in X_n$ can be contained in at most $N_\eta$ different elements $T\in\mathcal{T}_n$.
\end{proof}

\begin{claim}\label{freeintervalbound}
For each $n\geq n_0$ and each $T\in \mathcal{T}_n$,
$$ \sum_{C\in \mathcal{C}_n^1, c_n^1(C) = T} |T| \leq C_{\eta,\Delta} m^{-1}|T|.$$
In other words, the collection $\mathcal{C}_n^1$ and the map $c_n^1$ satisfy \eqref{connectable1}.
\end{claim}
\begin{proof}
Fix $n\geq n_0$ and $T\in\mathcal{T}_n$. By construction, $T$ is a lift of an arc $I$ in $e_{i} \subset \Gamma_i \subset X_i$, for some $i<n$.

Let $C_n\in\mathcal{C}^1_n$ have $c^1_n(C_n) = T$. The connection $C_n$ is a lift of a connection $C_{\ell+1}\in \mathcal{C}^1_{\ell+1}$, for some $\ell< n$, which was added in either Case A2 or Case B above. 

Suppose the connection $C_{\ell+1}$ was added to $\mathcal{C}^1_{\ell+1}$ in Case A2. In other words, $C_{\ell+1}$ connects two different lifts of a point $v'_\ell$ in a monotone component $P_\ell\subset \Gamma_\ell$ . Since we are in Case A2, $i=\ell-1$ and $T$ is defined as a lift of a sub-arc $I$ of a bad edge $e_{\ell-1}$ in $\Gamma_{\ell-1}$. Let $e_\ell$ be the edge of $P_\ell \subset \Gamma_\ell$ associated to the point $v'_\ell$ as in \eqref{env'n}. 

If $e_\ell$ projects into a bad edge $e_{\ell-1}$ of $\Gamma_{\ell-1}$, then $\pi_{\ell-1}(e_\ell)$ must contain the endpoint of $e_{\ell-1}$, by Lemma \ref{projectdata}.

Observe that since $e_\ell$ is contained in a monotone component as in \ref{otheredgeslift}, it must be that $\pi_{\ell-1}(e_\ell)$ is not contained in an edge with $\beta\geq \epsilon$. Therefore, Lemma \ref{onetoone} applies: By Lemma \ref{onetoone}, there are at most $2$ such edges $e_\ell$ (one for each endpoint of $e_{\ell-1}$), each of which can contain at most $2$ break points $v'_\ell$. Therefore in this case, the arc in $e_{\ell-1}$ is used at most $4$ times.

If $e_\ell$ projects into a good edge of $X_{\ell-1}$, then by Lemma \ref{projectdata}, $f_{\ell-1}$ is in $\Gamma_{\ell-1}$. Furthermore, by Lemma \ref{projectdata}, $e_\ell$ must both contain and project to the left of $a(f_{\ell-1})$, and $f_{\ell-1}$ is preceded by a bad edge $e_{\ell-1}$ in a monotone component of $\Gamma_{\ell-1}$. (Or similarly with $b(f_{\ell-1})$ on the other side.) The set $T$ is then be defined as a lift of an arc in $e_{\ell-1}$.

Thus, by Lemma \ref{onetoone}, there are at most $2$ such edges $e_\ell$, each of which can contain at most $2$ break points like $v'_\ell$ for which we need a connection. Therefore in this case, the arc $e_{\ell-1}$ is used at most $4$ times.

Therefore, if $T$ was originally defined in Case A2, there are at most $8$ connections that can use it, each of which have length at most $C_\eta m^{-\ell} \leq C_{\eta}m^{-1}|T|$. This completes the argument for Case A2.

Now suppose that $T$ was originally defined in Case B in a scale $\ell$. Let $j\geq 1$ be the integer defined in Case B above. If $j=1$, we can argue as in Case A2 of this lemma to show that $T$ is equal to $c^\ell_1(C)$ for at most $8$ different $C\in \mathcal{C}^1_\ell$. 

Otherwise suppose that $T$ was originally defined at a scale $\ell$ with $j>1$, as a lift of an arc in $e_{\ell-j+1}\subset\Gamma_{\ell-j+1}$. If $n\geq \ell$ and $C\in \mathcal{C}^1_n$ is a lift of a connection added originally to $\mathcal{C}^1_{\ell+1}$, then the only way that $c^1_n(C)$ can be $T$ is if there is a break point of level $\leq \ell-j+1$ within one edge of $e_{\ell-j+1}$ that both endpoints of $C$ project onto, and furthermore, the edges containing the endpoints of $C$ do not project into any edge $e_k$ with $\beta_E(Q(e_k))\geq \epsilon$ for $\ell-j+1\leq k < \ell$. Since there can be at most $6\Delta$ such break points near $e_{\ell-j+1}$, it follows by Lemma \ref{onetoone} that there can be at most $12\Delta^2$ such $C\in\mathcal{C}^1_n$ with $c^1_n(C)=T$.

Therefore, if $T$ was originally defined in Case B, there are at most $12\Delta^2+8$ connections that can use it, each of which have length at most $C_{\eta}m^{-1}|T|$. This completes the argument for Case B.
\end{proof}

\begin{claim}
For each $n\geq n_0$ and each $Q\in \mathcal{Q}$,
$$ \sum_{C\in \mathcal{C}_n^2, c_n^2(C) = Q} \mathcal{H}^1(C) \leq C_{X} \diam(Q).$$
Furthermore, if $Q\in c_n^2(\mathcal{C}_n^2)$, then $\beta_E(Q)\geq \epsilon$.

In other words, the collection $\mathcal{C}_n^2$ and the map $c_n^2$ satisfy \eqref{connectable2} and \eqref{connectable3}.
\end{claim}
\begin{proof}
There are two ways that a cube $Q=Q(e_i)\in \mathcal{Q}$ (where $e_i$ is an edge of $X_i$) can be in  $c_n^2(\mathcal{C}_n^2)$ for some $n\geq i\geq n_0$.

One way is by the argument given in Case A1 of Lemma \ref{connectable}. This can occur at most $C_X$ times for each $Q$, since it is used to initially define connections at scale $n=i+2$ that project into $e_i$, of which there are a controlled number. Each such connection has length bounded by $C_\eta m^{-i} \leq C_X \diam(Q)$.

The other way that a cube $Q=Q(e_i)\in \mathcal{Q}$ can be in  $c_n^2(\mathcal{C}_n^2)$ for some $n\geq i\geq n_0$ is by the argument at the beginning of Lemma \ref{connectable}. Let $Q=Q(e_i)\in Q$, where $e_i$ is an edge of $X_i$. Recall the definition of $V_n\subset V(\Gamma_n)$ from \ref{otheredgeslift}. A connection $C^2_n\in \mathcal{C}^2_n$ of this type having $c_n^2(C^2_n) = Q$ connects two vertices $v_n$ and $w_n$ of $V_n$ with the property that $i$ is the largest natural number $j$ such that $\pi_j(v_n)$ lies in an edge $e_j$ of $X_j$ with $\beta_E(Q(e_j))\geq \epsilon$. It follows from Lemma \ref{onetoone} that there are at most $C_X$ such pairs of vertices in $\Gamma_n$ and therefore at most $C_X$ elements of $\mathcal{C}_n^2$ that map to $Q(e_i)$ under $c^2_n$.

Finally, the fact that if $Q\in c_n^2(\mathcal{C}_n^2)$, then $\beta_E(Q)\geq \epsilon$ is clear from the construction.
\end{proof}

This completes the proof of Lemma \ref{connectable}.

\section{Sub-convergence to a rectifiable curve}\label{convergence}
Observe that the construction in Section \ref{liftingalgorithm} and Lemma \ref{connectable} now imply the following bound on lengths:
\begin{align*}
|\Gamma_n \cup \bigcup_{C^1 \in \mathcal{C}^1_n} C^1 \cup \bigcup_{C^2 \in \mathcal{C}^2_n} C^2| &\leq m^{-n_0} + C_X \sum_{Q\in\mathcal{Q}, \beta_E(Q)\geq\epsilon} \diam(Q) + \sum_{\mathcal{C}^1_n} |C^1| +  \sum_{\mathcal{C}^2_n} |C^2|\\
&\leq m^{-n_0} + C_X \sum_{Q\in\mathcal{Q}, \beta_E(Q)\geq\epsilon} \diam(Q) + \sum_{T\in\mathcal{T}_n}\sum_{c^1_n(C^1)=T} |C^1| +  \sum_{Q\in\mathcal{Q}, \beta_E(Q)\geq\epsilon} \sum_{c^2_n(C^2)=Q} |C^2|\\
&\leq m^{-n_0} + C_X \sum_{Q\in\mathcal{Q}, \beta_E(Q)\geq\epsilon} \diam(Q) + \sum_{T\in\mathcal{T}_n}C_{\eta}m^{-1}|T| +  \sum_{Q\in\mathcal{Q}, \beta_E(Q)\geq\epsilon} C_X \diam(Q)\\
&\leq m^{-n_0} + C_X \sum_{Q\in\mathcal{Q}, \beta_E(Q)\geq\epsilon} \diam(Q) + C_{\eta}m^{-1}N_{\eta,\Delta}|\Gamma_n|\\
&\leq m^{-n_0} + C_X \sum_{Q\in\mathcal{Q}, \beta_E(Q)\geq\epsilon} \diam(Q) + \frac{1}{100}|\Gamma_n|.
\end{align*}
In the last inequality above, we have used the assumption in \eqref{mlarge} that $m$ is large depending on $\eta$ and $\Delta$.

Hence
$$|\Gamma_n \cup \bigcup_{C^1 \in \mathcal{C}^1_n} C^1 \cup \bigcup_{C^2 \in \mathcal{C}^2_n} C^2| \leq 2m^{-n_0} + C_X \sum_{Q\in\mathcal{Q}, \beta_E(Q)\geq\epsilon} \diam(Q).$$
It follows that a subsequence of the sequence of continua
$$ \left\{\Gamma_n \cup \bigcup_{C^1 \in \mathcal{C}^1_n} C^1 \cup \bigcup_{C^2 \in \mathcal{C}^2_n} C^2 \right\} $$ 
converges (in the Gromov-Hausdorff sense) to a continuum $\Gamma$ in $X$. From standard results about Hausdorff convergence of connected sets (see, e.g., \cite{FFP07}, Theorem 5.1), the limit $\Gamma$ is a compact connected set  whose length satisfies
$$ |\Gamma| \leq 2m^{-n_0} + C_X \sum_{Q\in\mathcal{Q}, \beta_E(Q)\geq\epsilon} \diam(Q).$$
(Note that to obtain this conclusion, we may consider the Gromov-Hausdorff convergence of $\{X_n\}$ to $\{X\}$ as occuring as true Hausdorff convergence in some ambient compact metric space $Z$.)

\section{Decomposition of the complement and proof of Proposition \ref{constructionprop}}\label{propositionproof}
We now complete the proof of Proposition \ref{constructionprop}.

Fix our compact set $E\subset X$ as in Proposition \ref{constructionprop}. Apply the algorithm of Sections \ref{liftingalgorithm} and \ref{connectability} to obtain (disconnected) sets $\Gamma_n\subset X_n$ as well as a continuum $\Gamma\subset X$ (as in Section \ref{convergence}). 

Let
$$\mathcal{Q}_\Gamma = \{Q\in\mathcal{Q} : Q=Q(e_n), e_n\subset \Gamma_n, \beta_E(Q)\geq \epsilon\}.$$

Property \eqref{gammalength} of Proposition \ref{constructionprop} is then clear as above; the only cubes $Q$ appearing in the sum are those in $\mathcal{Q}_\Gamma$. Property \eqref{gammanearE} follows using Lemma \ref{projection} and the fact that $\pi_{n_0}(\Gamma_{n_0+1}) = e_{n_0}$. Property \eqref{gammaprojection} holds for similar reasons: if $Q(e_n)\in \mathcal{Q}_\Gamma$, then $e_n\subset \Gamma_n$ and hence by Lemma \ref{projection} $\pi_{n_0}(e_n)\subset e_{n_0}$.

We now write $E\setminus \Gamma$ as a union of sets $E(e_n)$ ($n\geq n_0+2$, $e_n$ an edge of $X_n$) in the following way:

For each $n\geq n_0$, let $\mathcal{E}_n$ be the collection of all edges $e_n$ in $X_n$ such that $\pi_{n_0}(e_n)\subset e_{n_0}$ and having the following property:
$$\dist(\pi_n(x),\Gamma_{n} \cap \pi_0^{-1}(\pi_0(x))) \geq m^{-n}$$
for some $x\in E \cap \pi_n^{-1}(e_n)$. Observe that, by construction, $\mathcal{E}_{n_0}$ and $\mathcal{E}_{n_0+1}$ are both empty. Furthermore, since $\Gamma$ contains a sub-sequential limit of the compacta $\Gamma_n$, each point of $E\setminus \Gamma$ is contained in some $E(e_n)$.

Let $\mathcal{G}_n \subset \mathcal{E}_n$ be the set of all edges $e_n \in \mathcal{E}_n$ such that $\pi_\ell(e_n)$ is not contained in an edge of $\mathcal{E}_\ell$ for any $\ell<n$. 

For each edge $e_n\in\mathcal{G}_n$, let $E(e_n) =  E \cap \pi_n^{-1}(e_n)$. (Otherwise set $E(e_n)=\emptyset$.) The union of all the sets $\{E(e_n)\}_{n\geq n_0+2, e_n\in\mathcal{G}_n}$ contains $E\setminus \Gamma$. Indeed, if $x\in E\setminus\Gamma$, then $\pi_n(x)$ must be contained in an edge of $\mathcal{E}_n$ for some $n$, and therefore in an edge of $\mathcal{G}_n$ for some $n$.

Observe that if $e_n\in\mathcal{G}_n$, then $\pi_{n_0}(e_n)\subseteq e_{n_0}$, simply because $e_n$ contains points of $\pi_n(E)$ and $\pi_{n_0}(E)\subset e_{n_0}$. The rest of property (\ref{bubblenotingamma}) is proven in the following claim:
\begin{claim}
If $Q(e_{\ell})\in \mathcal{Q}_\Gamma$, $n\leq \ell$, and $e_n\in\mathcal{G}_n$, then $\pi_n(e_\ell)\not\subseteq e_n$.
\end{claim}
\begin{proof}
If $Q(e_{\ell})\in \mathcal{Q}_\Gamma$, then $e_\ell\subset \Gamma_\ell$ for some $\ell\geq n$. Suppose that $\pi_n(e_\ell)\subseteq e_n$, where $e_n\in \mathcal{G}_n$. It follows from the property \eqref{goodedgeprojection} of the sets $\{\Gamma_i\}$ (from Section \ref{liftingalgorithm}) that there is an edge $f_n\subset \Gamma_n$ with the same endpoints as $e_n$. But then this violates the assumption that $e_n\in \mathcal{G}_n \subset \mathcal{E}_n$, because by assumption, $\pi_n(E)$ does not contain the midpoint of $e_n$, and therefore every point $x_n$ of $\pi_n(E) \cap e_n$ satisfies
$$\dist(x_n,\Gamma_{n} \cap \pi_0^{-1}(\pi_0(x_n))) < \dist(x_n,f_{n} \cap \pi_0^{-1}(\pi_0(x_n))) < m^{-n}.$$
\end{proof}

Property \eqref{bubblesdisjoint} of Proposition \ref{constructionprop} is immediate from the definition of $\mathcal{G}_n$. To see property \eqref{bubbleneargamma}, note that if $e_n\subset \mathcal{G}_n$ and $x\in E(e_n)$, then $\dist(\pi_{n-1}(x), \Gamma_{n-1}) \leq m^{-(n-1)}$. Property \eqref{goodedgeprojection} and Lemma \ref{projection} then imply that $\pi_{n-1}(\Gamma)$ contains a point within distance $2m^{-(n-1)}$ of $x$, and so \eqref{bubbleneargamma} follows from Lemma \ref{piproperties2}.

It remains only to verify property \eqref{bubblesum} of Proposition \ref{constructionprop}. To do so, we do using the following claim.
\begin{claim}\label{bubblefreebeta}
Fix $n\geq n_0$ and $e_n\in\mathcal{G}_n$. At least one of the following two options must hold:
\begin{enumerate}[(a)]
\item \label{bubble_free}
We have $n\geq n_0+3$ and there is a bad edge $g_{n-3}$ in $\Gamma_{n-3}$ with $\beta_E(Q(g_{n-3}))<\epsilon$ such that
$$ \dist( \pi_{n-3}(e_n), f_{n-3}) \leq m^{n-3}. $$

\item \label{bubble_beta} For some $0\leq \ell\leq 5$, there exists $Q\in \mathcal{Q}_{n-\ell}$ with $\beta_E(Q)\geq\epsilon$ such that $E(e_n)\subseteq Q$.
\end{enumerate}
Furthermore, at most $C_\Delta m$ edges of $\mathcal{G}_n$ that fall into case \eqref{bubble_free} are associated to each such edge $f_{n-3}$ in $\Gamma_{n-3}$.
\end{claim}
\begin{proof}
First, we observe that if $e_n \in \mathcal{G}_n$ and $n\leq n_0+2$, then $\beta_E(Q(e_{n_0}))\geq \epsilon$ and therefore \eqref{bubble_beta} holds. Indeed, if $n\leq n_0+2$ and $\beta_E(Q(e_{n_0}))<\epsilon$, then it is clear from the construction that
$$\dist(\pi_n(x), \Gamma_n \cap \pi_0^{-1}(\pi_0(x))) \leq 2A\epsilon m^{-n_0} < m^{-n}/2, $$
and therefore that $\mathcal{G}_n$ is empty.

So we now suppose that $e_n\in\mathcal{G}_n$, that $n\geq n_0+3$, and that \eqref{bubble_beta} does not hold. Let $e_{n-1}$ contain $\pi_{n-1}(e_n)$. By definition of $\mathcal{G}_n$, it must be that $e_{n-1}$ shares at least one vertex with an edge $f_{n-1}$ of $\Gamma_{n-1}$ such that $\pi_0(f_{n-1})=\pi_0(e_{n-1})$. (It may be that $f_{n-1}=e_{n-1}$.) In addition, the fact that $e_n\in\mathcal{G}_n$ implies that $f_{n-1}$ can be chosen to satisfy the hypotheses of Lemma \ref{projectdata} (at scale $n-1$).

Note that $\beta_E(Q(f_{n-1}))<\epsilon$ by our assumption that \eqref{bubble_beta} does not hold. Thus, since $Q(f_{n-1})$ contains $e_{n-1}$, there are at most $\Delta m$ edges $e_n\in \mathcal{G}_n$ associated to $f_{n-1}$ in this way. (All of these edges must either be in or adjacent to a single monotone geodesic in $X_n$.)

Let $f_{n-2}$ be the edge of $X_{n-2}$ containing $\pi_{n-2}(f_{n-1})$. Consider the possibilities for $f_{n-2}$ outlined in Lemma \ref{projectdata}. By a similar argument as given in the first paragraph of this lemma (using $\epsilon < 2Am^{-10}$), it must be that either $f_{n-2}$ is a bad edge of $\Gamma_{n-2}$ or is adjacent to a bad edge of $\Gamma_{n-2}$, and furthermore there are at most $4$ locations on $f_{n-2}$ into which $f_{n-1}$ could project. In addition, $f_{n-2}$ satisfies the conditions of Lemma \ref{projectdata} (at scale $n-2$).

Similarly, the edge $f_{n-3}$ of $X_{n-3}$ containing $\pi_{n-3}(f_{n-2})$ is either a bad edge of $\Gamma_{n-3}$ or is adjacent to a bad edge of $\Gamma_{n-3}$, and furthermore there are at most $4$ locations on $f_{n-3}$ (or an edge sharing the same endpoints) into which $f_{n-2}$ could project.

There are therefore at most $16$ locations on $f_{n-3}$ (or an edge sharing the same endpoints) into which $f_{n-1}$ could project. Because of our assumption that \eqref{bubble_beta} does not hold, Lemma \ref{onetoone} implies that there are therefore at most $16$ edges $f_{n-1}$ as above that can project into $f_{n-3}$. Therefore there are at most $16\Delta m$ possible edges $e_n\in \mathcal{G}_n$ associated to the bad edge $f_{n-3}$. 

Let $g_{n-3}$ be the bad edge of $\Gamma_{n-3}$ which is either equal to or adjacent to $f_{n-3}$. Since $g_{n-3}$ may be adjacent to up to $2\Delta$ other edges like $f_{n-3}$, there are at most $32\Delta^2 m$ edges $e_n\in\mathcal{G}_n$ associated to $g_{n-3}$ in this way. This completes the proof.
\end{proof}

Property (\ref{bubblesum}) is now proven as follows: Fix any $N\geq n_0$. In each edge $g_{n-3}$ ($n\leq N$) used in option (\ref{bubble_free}) of Claim \ref{bubblefreebeta}, form an arc $I$ with
$$ m^{-{n+3}}/10 \leq |I| \leq m^{-(n+3)}$$
that lifts (by Lemma \ref{edgeslift}) to a set $T$ in $\Gamma_N$. As in Cases A2 and B of Section \ref{connectability}, these sets can be chosen to be $N_{\eta}$-overlapping. Observe also that each $Q\in \mathcal{Q}$ with $\beta_E(Q)\geq\epsilon$ can contain at most $C_X$ different sets $E(e_n)$, as in the second option in Claim \ref{bubblefreebeta}. Therefore, using Claim \ref{bubblefreebeta},
\begin{align*}
\sum_{n= n_0}^N\sum_{e_n\in \mathcal{G}_n} |e_n| &\leq \sum_{n= n_0}^N\sum_{e_n\in \mathcal{G}_n, (\ref{bubble_free})} |e_n| +  \sum_{n= n_0}^N\sum_{e_n\in \mathcal{G}_n, (\ref{bubble_beta})} |e_n|\\
&\leq C_{\eta,\Delta} m m^{-3}|\Gamma_N| + C_{X}\sum_{Q\subset \mathcal{Q}_\Gamma, \beta_E(Q)\geq\epsilon} \diam(Q)\\
&\leq C_{\eta,\Delta} m m^{-3}|\Gamma| + C_{X}\sum_{Q\subset \mathcal{Q}_\Gamma, \beta_E(Q)\geq\epsilon} \diam(Q),
\end{align*}
where the summation on the right side of the first line is broken based on which case of Claim \ref{bubblefreebeta} the edge $e_n$ falls into.

Letting $N$ tend to infinity completes the proof of Proposition \ref{constructionprop}.

\section{Proof of Theorem \ref{construction}}\label{theoremproof}

Without loss of generality, we may assume that $\pi_{n_0}(E)$ is contained in a single edge of $X_{n_0}$, where
$$ m^{-(n_0+1)}<\diam(E) \leq m^{-n_0}.$$
Indeed, suppose we can prove Theorem \ref{construction} in this case. For a general set $E$, $\pi_{n_0}(E)$ is contained in a union of at most $\Delta$ edges of $X_{n_0}$. We can then apply Theorem \ref{construction} to each of the $\Delta$ subsets of $E$ projecting into each edge, and then join each of those $\Delta$ curves by connections of total length at most $C_{\eta,\Delta} m^{-n_0} \leq C_{X}\diam(E)$, using Lemma \ref{piproperties}.

We will make one other convenient reduction: Without loss of generality, we can also assume that, for every $n\in\mathbb{N}$, $\pi_n(E)$ does not contain any vertex or any edge-midpoint of $X_n$. (Indeed, by standard arguments we can assume that $E$ is finite, and then perturb by arbitrarily small amounts to avoid such points.)

With the above assumptions in place, we can apply Proposition \ref{constructionprop} to $E$ to obtain a continuum $\Gamma$, a collection $\mathcal{Q}_\Gamma$, and a partition $\{E(e_n)\}$ of $E\setminus \Gamma$. Arbitrarily label those $E(e_n)$ in the partition which are non-empty by $E_1, E_2, \dots$. (There are countably many, possibly finitely many, of these sets.)

Apply Proposition \ref{constructionprop} to each $E_i$ to generate a continuum $\Gamma_i$ and a partition of $E_i\setminus \Gamma_i$ into sets $\{E_{i,j}\}_j$.

After $k$ iterations of this process, we have constructed connected sets $\{\Gamma_\alpha\}$, where each $\alpha$ is a string of integers of length $|\alpha|<k$, and $E\setminus \cup_\alpha \Gamma_\alpha$ is a union of disjoint sets $\{E_\beta\}$, where each $\beta$ is a string of integers of length $|\beta|=k$ whose first $k-1$ integers forms a substring corresponding to some $\Gamma_\alpha$. (We consider the original set $E$ as $E_\emptyset$ and the first curve $\Gamma$ as $\Gamma_\emptyset$.)

By Proposition \ref{constructionprop}, each $\Gamma_\alpha$ corresponds to an edge $e_{n_\alpha}\subset X_{n_\alpha}$ and satisfies
$$ |\Gamma_\alpha| \leq 2m^{-n_\alpha} + \sum_{Q\subset \mathcal{Q}_{\Gamma_\alpha}, \beta_E(Q)\geq \epsilon} \diam(B) $$
for some subcollection $\mathcal{Q}_{\Gamma_\alpha}\subset \mathcal{Q}$.

\begin{claim}\label{Qgammadisjoint}
Let $\alpha$ and $\beta$ be two distinct strings of integers such that $\Gamma_\alpha$ and $\Gamma_\beta$ are in our collection. Then $\mathcal{Q}_{\Gamma_\alpha}$ and $\mathcal{Q}_{\Gamma_\beta}$ are disjoint subcollections of $\mathcal{Q}$.
\end{claim}
\begin{proof}
First suppose that neither $\alpha$ and $\beta$ is a prefix of the other. Let $\gamma$ denote the longest common initial substring of $\alpha$ and $\beta$, so that 
$$ \alpha = \gamma i \alpha' \text{ and } \beta = \gamma j \beta' $$
for integers $i\neq j$ and strings $\alpha'$, $\beta'$.

By our construction and property \eqref{bubblesdisjoint} of Proposition \ref{constructionprop}, this means that there are edges $e_{n_{\gamma i}}$ and $e_{n_{\gamma j}}$ (in some $X_{n_{\gamma i}}$ and $X_{n_{\gamma j}}$, respectively) such that neither $e_{n_{\gamma i}}$ nor $e_{n_{\gamma j}}$ project into the other and such that every edge $e$ with $Q(e)\in \mathcal{Q}_{\Gamma_\alpha}$ projects into $e_{n_{\gamma i}}$ and every edge $e$ with $Q(e)\in \mathcal{Q}_{\Gamma_\beta}$ projects into $e_{n_{\gamma j}}$. It then follows that no $Q(e)$ can be in both $\mathcal{Q}_{\Gamma_\alpha}$ and $\mathcal{Q}_{\Gamma_\beta}$.

Now suppose that $\alpha$ is a prefix of $\beta$, i.e. that $\beta = \alpha i \beta'$ for some integer $i$ and some (possibly empty) string $\beta'$. Consider any $Q=Q(e)\in \mathcal{Q}_{\Gamma_\beta}$. Let $e_{n_\beta}$ be the edge (in some graph $X_{n_\beta}$) such that $E_\beta=E(e_{n_\beta})$ when Proposition \ref{constructionprop} was applied to the previous step. Then $\pi_{n_\beta}(e)\subset e_{n_\beta}$, by property \eqref{bubblenotingamma} of Proposition \ref{constructionprop}. 

Let $e_{n_{\alpha i}}$ be the edge such that $E_{\alpha i}=E(e_{n_{\alpha i}})$ when Proposition \ref{constructionprop} was applied to $\Gamma_{\alpha}$. By repeated applications of property \eqref{bubblenotingamma}, we see that $\pi_{n_{\alpha i}}(e_{n_\beta})\subset e_{n_{\alpha i}}$. Therefore
$$ \pi_{n_{\alpha i}}(e) \subset \pi_{n_{\alpha i}}(e_{n_\beta})\subset e_{n_{\alpha i}}.$$

On the other hand, again by \eqref{bubblenotingamma}, if $Q(e)$ were in $\mathcal{Q}_{\Gamma_\alpha}$, then $e$ could not project into $e_{n_{\alpha i}}$.

Therefore $\mathcal{Q}_{\Gamma_\beta}$ and $\mathcal{Q}_{\Gamma_\alpha}$ are disjoint.

\end{proof}

Now for each string $\alpha\neq\emptyset$, let $C_\alpha$ be a continuum of length at most $2m^{-(n_\alpha-1)}=2m|e_{n_\alpha}|$ joining $\Gamma_{\alpha}$ to $\Gamma_\beta$, where $\beta$ is the prefix of $\alpha$. Such a continuum exists by (\ref{gammanearE}) and (\ref{bubbleneargamma}).

Thus, given a fixed string $\beta$, we have, using properties \eqref{gammalength} and \eqref{bubblesum} of Proposition \ref{constructionprop}, that
\begin{align}
\sum_{i\in\mathbb{N}} |\Gamma_{\beta i} \cup C_{\beta i}| &\leq \sum_i 2m^{-n_{\beta i}} + C_X \sum_i \sum_{Q\in \mathcal{Q}_{\Gamma_{\beta i}}} \diam(Q) + \sum_i m|e_{n_{\beta i}}|\\
\label{lengthalign}&\leq C_X \sum_i \sum_{Q\in \mathcal{Q}_{\Gamma_{\beta i}}} \diam(Q) + C_X\sum_{Q\in \mathcal{Q}_{\Gamma_\beta}} \diam (Q) + C_{\eta,\Delta}m^{-1}|\Gamma_\beta|.
\end{align}

Now let $G_0 = \Gamma_\emptyset$. Inductively define a compact connected set $G_k\subset X$ by
$$ G_k = G_{k-1} \cup \bigcup_{|\alpha| = k} (\Gamma_\alpha \cup  C_\alpha).$$
It follows from \eqref{lengthalign} and our choice of $m$ that (treating $G_{-1}=\emptyset$),
$$ |G_{k+1}\setminus G_k| \leq C_X\left(\sum_{|\beta|=k}\sum_{Q\in\mathcal{Q}_{\Gamma_\beta}, \beta_E(Q)\geq \epsilon} \diam (Q) + \sum_{|\alpha|=k+1}\sum_{Q\in\mathcal{Q}_{\Gamma_\alpha}, \beta_E(Q)\geq \epsilon} \diam (Q)\right) + \frac{1}{100}|G_k \setminus G_{k-1}|.$$

From this, Proposition \ref{constructionprop} \eqref{gammalength}, and Claim \ref{Qgammadisjoint}, it follows that
$$ |G_k| \leq C_X\left(\diam(E) + \sum_{B\subset \mathcal{B}, \beta_E(B)\geq\epsilon} \diam(B)\right)$$
for all $k$.

Now, from standard results about Hausdorff convergence of connected sets (see, e.g., \cite{FFP07}, Theorem 5.1), there is a subsequence of $\{G_k\}$ that converges in the Hausdorff sense to a compact connected set $G\subset X$ whose length is bounded as in Theorem \ref{construction}.

It follows from Proposition \ref{constructionprop} (the fact that $n\geq n_0+2$, (\ref{bubblenotingamma}) and (\ref{bubbleneargamma})) that, for each $k\geq 1$, $E$ is contained in the $C_\eta m^{-n_0-k}$-neighborhood of $G_k$. Hence, $G$ contains $E$.

\bibliography{TSTbib}{}
\bibliographystyle{plain}

\end{document}